\documentclass{article}

\usepackage[utf8]{inputenc}
\usepackage{amsmath, authblk}
\usepackage{amssymb,comment}
\usepackage{amsthm}
\usepackage{aligned-overset}
\usepackage{enumitem}
\usepackage{xcolor}
\usepackage{graphicx}

\allowdisplaybreaks
\theoremstyle{plain}
\newtheorem{thm}{Theorem}[section]
\newtheorem{lem}[thm]{Lemma}
\newtheorem{prop}[thm]{Proposition}
\newtheorem{cor}[thm]{Corollary}
\theoremstyle{definition}
\newtheorem{defi}[thm]{Definition}
\newtheorem{exmp}[thm]{Example}
\newtheorem{rem}[thm]{Remark}

\newcommand{\order}[1]{\xrightarrow[{\raisebox{.4ex}[0pt][0pt]{$\scriptstyle #1$}}]{\; o \; }}
\newcommand{\R}{\mathbb{R}}
\newcommand{\N}{\mathbb{N}}

\usepackage[
  backend=biber,
  style=numeric,
  sorting=nyt,
  giveninits=true,
  maxnames=99,
  maxbibnames=99,
  doi=false,
  isbn=false,
  url=false,
  eprint=true
]{biblatex}

\DeclareFieldFormat[article,inproceedings,misc]{title}{#1}

\DeclareFieldFormat[article,inproceedings]{pages}{#1}

\DeclareBibliographyDriver{article}{%
  \printnames{author}%
  \newunit\newblock%
  \printfield{title}%
  \newunit\newblock%
  \printfield{journaltitle}%
  \setunit{\addcomma\space}%
  \printfield{volume}%
  \iffieldundef{number}%
    {}%
    {\mkbibparens{\printfield{number}}}%
  \setunit{\addcolon\allowbreak}%
  \printfield{pages}%
  \setunit{\addcomma\space}%
  \printfield{year}%
  \finentry%
}

\DeclareBibliographyDriver{inproceedings}{%
  \printnames{author}%
  \newunit\newblock%
  \printfield{title}%
  \newunit\newblock%
  \printfield{booktitle}%
  \setunit{\addcomma\space}%
  \printfield{pages}%
  \setunit{\addcomma\space}%
  \printfield{year}%
  \finentry%
}

\DeclareBibliographyDriver{book}{%
  \printnames{author}%
  \newunit\newblock%
  \printfield{title}%
  \newunit\newblock%
  \printfield{edition}%
  \newunit\newblock%
  \printlist{publisher}%
  \setunit{\addcomma\space}%
  \printfield{year}%
  \finentry%
}

\DeclareBibliographyDriver{misc}{%
  \printnames{author}%
  \newunit\newblock%
  \printfield{title}%
  \newunit\newblock%
  \printfield{year}%
  \newunit\newblock%
  \usebibmacro{eprint}%
  \finentry%
}

\title{Bipolar Theorems for Sets of Non-negative Random Variables\footnote{funded by the Deutsche Forschungsgemeinschaft (DFG, German Research Foundation) – 471178162}}
\author{Johannes Langner\footnote{johannes.langner@insurance.uni-hannover.de} \, \& Gregor Svindland\footnote{gregor.svindland@insurance.uni-hannover.de}}
\affil{\small Institute of Actuarial and Financial Mathematics and House of Insurance\\ Gottfried Wilhelm Leibniz Universit\"at Hannover}
\date{01.10.2025}

\begin{document}

\maketitle

\begin{abstract}
    This paper assumes a robust, in general not dominated, probabilistic framework and provides necessary and sufficient conditions for a bipolar representation of subsets of the set of all quasi-sure equivalence classes of non-negative random variables, without any further conditions on the underlying measure space. This generalizes and unifies existing bipolar theorems proved under stronger assumptions on the robust framework. Applications are in areas of robust financial modeling. \\
    {\bf Keywords:} robust financial models, non-dominated set of probability measures, bipolar theorem, sensitivity, convergence and closure on robust function space \\
    {{\bf MSC2020}: 46A20, 46E30, 46N10, 46N30, 60B11, 91G80} \\
    {{\bf JEL}: C65, D80}
\end{abstract}

\section{Introduction}

The well-known bipolar theorem proved in \cite{BS1999} provides necessary and sufficient conditions for the existence of a bipolar representation of a set $\mathcal{C} \subseteq L^{0}_{P+}$ by means of elements of $L^{0}_{P+}$ itself, see Theorem~\ref{thm:BS}. Here, $L^{0}_{P+} := L^{0}_{+}(\Omega, \mathcal{F}, P)$ denotes the positive cone of $L^{0}(\Omega, \mathcal{F}, P)$, which is endowed with the topology induced by convergence in probability, and $(\Omega, \mathcal{F}, P)$ is a probability space. An important application of this result is the dual characterization of solutions to utility maximization problems, see \cite{KS1999, KS2003}.

\cite{BS1999} show that $\mathcal{C}$ allows a bipolar representation in $L^{0}_{P+}$ if and only if $\mathcal{C}$ is convex, solid, and closed in probability. The aim of this paper is to generalize this result to a so-called robust framework, where the probability measure $P$ is replaced by a family of probability measures $\mathcal{P}$ which is not necessarily dominated. Such extensions have already been studied in \cite{BK2019}, \cite{GM2020}, and \cite{ LMS2022} where sufficient conditions for a bipolar representation in very particular robust frameworks are given. 

In this paper, without further assumptions on the underlying measure space, we provide necessary and sufficient conditions for a bipolar representation of $\mathcal{C} \subseteq L^{0}_{c+}$, where $c$ is the upper probability induced by the set of probability measures $\mathcal{P}$, and $L^{0}_{c+}$ denotes the robust counterpart of $L^{0}_{P+}$. As a byproduct we obtain a common framework for and unify the already mentioned bipolar results of  \cite{ BK2019}, \cite{ BS1999}, \cite{GM2020}, and \cite{ LMS2022}.

Of course, convexity and solidity of $\mathcal C$ are necessary conditions for a bipolar representation, also in robust frameworks. A key observation, however, is that any bipolar representation requires a property called \emph{$\mathcal{P}$-sensitivity}  of $\mathcal C$, see Definition~\ref{def:sensitivity}, a property which is trivially satisfied in the classical dominated case and only reveals itself in a non-dominated robust framework. $\mathcal{P}$-sensitivity as an essential property to handle sets of random variables over robust model space also appeared in \cite{BM2020} and \cite{MMS2018}. Essentially, a set is $\mathcal{P}$-sensitive if membership to that set is determined by separate evaluations under (all) probability measures which are absolutely continuous to $\mathcal{P}$. This allows a 'localizing' strategy when deriving results for such sets, where localizing means verifying some statement separately under each probability measure. Not only is $\mathcal{P}$-sensitivity necessary for any bipolar representation, but the mentioned localizing strategy allows to \textit{lift} bipolar theorems known within a dominated framework to the robust model space, see Sections~\ref{sec:lifting} and \ref{sec:bip:thm}. Indeed, in this way we derive our main results on bipolar representations for sets $\mathcal{C}\subseteq L^0_{c+}$ by lifting the mentioned bipolar theorem of \cite{BS1999} from $L^0_{P+}$ to $L^0_{c+}$. In fact, we will pay more attention to lifting a related bipolar theorem  given in \cite{KS2011} which has the advantage of the more manageable dual space $L^\infty_P$, see Section~\ref{sec:bip:thm}.

Regarding $\mathcal{P}$-sensitivity, we give a number of conditions that guarantee $\mathcal{P}$-sensitivity of sets in Section~\ref{sec:pSensi:rel}. In particular, we will show that $\mathcal{P}$-sensitivity is equivalent to the aggregation property known from robust statistics, see, e.g., \cite[Section 1.5]{T1991}, or robust stochastic (financial) models, see, e.g., \cite{STZ2011}. 

Let us revisit the list of necessary properties for a bipolar representation. It is clear that this list must also include some kind of closedness of $\mathcal C$. In this respect it turns out that, in contrast to dominated models where closedness with respect to convergence in probability under the dominating probability measure is the canonical choice, in the non-dominated case there are a variety of notions of closedness which offer themselves as necessary and reasonable requirements depending on the point of view on the problem.  For a solid set, all of them may be seen as robust generalizations of closedness with respect to convergence in probability.
A main contribution of this paper is to relate the underlying notions of convergence on $L^0_c$, and thus the different types of closedness, to each other, see Section~\ref{sec:clsd}. Eventually, we identify sequential order closedness with respect to the quasi-sure order as the appropriate equivalent of a number of notions of closedness for solid sets which are necessary, and in fact sufficient in combination with the other properties mentioned above, for a bipolar representation of $\mathcal C$.

In Section~\ref{sec:bip:thm}, we collect versions of the bipolar theorem on $L^0_{c+}$ obtained by the lifting procedure described earlier. The versions differ in the choice of the class of dual elements. Appropriate classes of dual elements turn out to be combinations of probability measures and test functions, or simply the set of finite measures. In Section~\ref{sec:appl}, we collect several applications of the bipolar representations given in Section~\ref{sec:bip:thm}. In particular we show how our results generalize the bipolar theorems of \cite{BK2019}, \cite{GM2020}, and \cite{LMS2022}. Moreover, we sketch applications to mathematical finance which are part of ongoing research, and finally we provide a mass transport type duality adopted from \cite{BK2019}. 

A widely studied source of uncertainty leading to robust models is uncertainty about the volatility of (continuous time) price processes, see, e.g., \cite{C2012} and \cite{STZ2011}. In that respect \cite{DHP2011} investigate capacities and robust function spaces based on sublinear expectations, in particular $G$-expectation. \cite{BK2012} study risk measures under model uncertainty with a focus on dual representations. Robust duality has further been explored by \cite{BD2018}, who examine general equilibrium theory under Knightian uncertainty. \cite{LN2022} study robust Orlicz spaces and their duals. As regards the robust bipolar theorems mentioned above, \cite{BK2019} prove a pointwise bipolar theorem within a model-free setting, see Section~\ref{sec:appli:3}. Applications of this result, as provided in \cite{BK2019}, include robust hedging in discrete time and the aforementioned mass transport type duality result. In their investigation of surplus-invariant risk measures, \cite{GM2020} derive another robust bipolar theorem, see Section~\ref{sec:appli:1}. Lastly, \cite{LMS2022} study model uncertainty from a reverse perspective, trying to understand which conditions the probabilistic model has to satisfy in order to obtain robust analogues of useful properties of the model space known in dominated frameworks. In this regard the bipolar theorem discussed in Section~\ref{sec:appli:2} appears.

The paper is organized as follows. In Section~\ref{sec:prelim:not}, we introduce some notation and preliminary results, including a first discussion of $\mathcal{P}$-sensitivity. Section~\ref{sec:bip:rep} recalls the bipolar theorems of \cite{BS1999} and \cite{KS2011}, and provides a reverse study of bipolar representations which already establishes the necessity of $\mathcal{P}$-sensitivity. In Section~\ref{sec:clsd}, we discuss the mentioned different concepts of closedness of sets in $L^0_c$. Then, in Section~\ref{sec:pSensi:rel}, we give conditions which imply $\mathcal{P}$-sensitivity of sets, and we also provide an equivalent characterization of $\mathcal{P}$-sensitivity in terms of the aggregation property mentioned above. Section~\ref{sec:bip:thm} collects our main results, which are versions of the bipolar theorem under uncertainty. Lastly, in Section~\ref{sec:appl}, we collect some applications.

\section{Preliminaries and Notation} \label{sec:prelim:not}
\subsection{Basics}\label{sec:basics}

Throughout this paper $(\Omega, \mathcal{F})$ denotes an arbitrary measurable space. By $ca$ we denote the real vector space of all countably additive finite variation set functions $\mu \colon \mathcal{F} \rightarrow \mathbb{R}$, and by $ca_+$ its positive elements ($\mu\in ca_+ \Leftrightarrow \forall A \in \mathcal{F} \ \mu[A] \geq 0$), that is all finite measures on $(\Omega, \mathcal{F})$.
Given non-empty subsets $\mathfrak{G}$ and $\mathfrak{I}$ of $ca_+$, we say that $\mathfrak{I}$ dominates $\mathfrak{G}$ ($\mathfrak{G} \ll \mathfrak{I}$) if for all $N \in \mathcal{F}$ satisfying $\sup_{\nu \in \mathfrak{I}}  \nu[N] = 0$, we have $\sup_{\mu \in \mathfrak{G}}  \mu[N] = 0$.  $\mathfrak{G}$ and $\mathfrak{I}$ are equivalent ($\mathfrak{G} \approx \mathfrak{I}$) if $\mathfrak{G} \ll \mathfrak{I}$ and $\mathfrak{I} \ll \mathfrak{G}$. For the sake of brevity, for $\mu \in ca_+$ we shall write $\mathfrak{G} \ll \mu$, $\mu \ll \mathfrak{I}$, and $\mu \approx \mathfrak{G}$ instead of $\mathfrak{G} \ll \{\mu\}$, $\{\mu\} \ll \mathfrak{I}$, and $\{\mu\} \approx \mathfrak{G}$, respectively. \\
$\mathfrak{P}(\Omega) \subseteq ca_+$ denotes the set of probability measures on $(\Omega, \mathcal{F})$ and the letters $\mathcal{P}$ and $\mathcal{Q}$ are used to denote non-empty subsets of $\mathfrak{P}(\Omega)$. 
Fix such a set $\mathcal{P}$. We then write $c$ for the induced upper probability $c \colon \mathcal{F} \rightarrow [0, 1]$ defined by
\begin{equation*}
    c[A] := \sup_{P \in \mathcal{P}} P[A], \quad A\in \mathcal{F}.
\end{equation*}
An event $A \in \mathcal{F}$ is called $\mathcal{P}$-polar if $c[A] = 0$. A property holds $\mathcal{P}$-quasi surely (q.s.)\ if it holds outside a $\mathcal{P}$-polar event. We set $ca_{c} := \{\mu\in ca \colon \mu \ll \mathcal P\}$, $ca_{c+} :=ca_{+}\cap ca_{c}$, and $\mathfrak{P}_c(\Omega):=\mathfrak{P}(\Omega)\cap ca_{c}$. 

Consider the $\R$-vector space $\mathcal{L}^{0} := \mathcal{L}^{0}(\Omega, \mathcal{F})$ of all real-valued random variables $f \colon \Omega \rightarrow \R$ as well as its subspace $\mathcal{N} := \{f \in \mathcal{L}^{0} \colon c[\lvert f \rvert > 0] = 0\}$. The quotient space $L^{0}_{c} := \mathcal{L}^{0}/\mathcal{N}$ consists of equivalence classes $X$ of random variables up to $\mathcal{P}$-q.s.\ equality comprising representatives $f \in X$. The equivalence class induced by $f\in \mathcal{L}^0$ in $L^{0}_{c}$ is denoted by $[f]_c$. The space $L^{0}_{c}$ carries the so-called $\mathcal{P}$-quasi-sure order $\preccurlyeq_{\mathcal P}$ as a natural vector space order: $X, Y \in L^{0}_{c}$ satisfy $X \preccurlyeq_{\mathcal P} Y$ if for $f \in X$ and $g \in Y$, $f \leq g$ $\mathcal{P}$-q.s., that is $\{f>g\}$ is $\mathcal P$-polar. In order to facilitate the notation, we suppress the dependence of $\preccurlyeq_{\mathcal P}$ on $\mathcal{P}$ and simply write $\preccurlyeq$ if there is no risk of confusion.  
$(L^{0}_{c}, \preccurlyeq)$ is a vector lattice, and for $X, Y \in L^{0}_{c}$, $f \in X$, and $g \in Y$, the minimum $X \wedge Y$ is the equivalence class $[f \wedge g]_{c}$ generated by the pointwise minimum $f \wedge g$, whereas the maximum $X \vee Y$ is the equivalence class $[f \vee g]_{c}$ generated by the pointwise maximum $f \vee g$. For an event $A \in \mathcal{F}$, $\chi_{A}$ denotes the indicator of the event (i.e. $\chi_{A}(\omega) = 1$ if and only if $\omega \in A$, and $\chi_{A}(\omega) = 0$ otherwise) while $\mathbf{1}_{A} := [\chi_{A}]_{c}$ denotes the generated equivalence class in $L^{0}_{c}$. Throughout the paper, for convenience, we identify the constants $m\in \R$ with the (equivalence classes of) constant random variables they induce, in particular $m=[m]_c=m\cdot \mathbf{1}_\Omega$. \\
A subspace of $L^{0}_{c}$ which will turn out to be important for our studies is the space $L^{\infty}_{c}$ of equivalence classes of $\mathcal{P}$-q.s.\ bounded random variables, i.e.,
\begin{equation*}
    L^{\infty}_{c} := \{X \in L^{0}_{c} \colon \exists m > 0 \ \lvert X \rvert \preccurlyeq m\}.
\end{equation*}
$L^{\infty}_{c}$ is a Banach lattice when endowed with the norm
\begin{equation*}
    \lVert X \rVert_{L^{\infty}_{c}} := \inf \{ m > 0 \colon \lvert X \rvert \preccurlyeq m\}, \quad X \in L^{0}_{c}.
\end{equation*}
$L^{0}_{c+}$ and $L^{\infty}_{c+}$ denote the positive cones of $L^{0}_{c}$ and $L^{\infty}_{c}$, respectively. If $\mathcal{P} = \{P\}$ is given by a singleton and thus $c=P$, we write $L^{0}_{P}$, $L^{\infty}_{P}$, and $[f]_P$ instead of $L^{0}_c$, $L^{\infty}_{c}$, and $[f]_{c}$, and similarly for other expressions where $c$ appears. Also, the $\mathcal P$-q.s.\ order in this case is the $P$-almost-sure (a.s.)\ order which we will also denote by $\leq_P$ when we are working with both the $\mathcal P$-q.s.\ order $\preccurlyeq$ for some set $\mathcal P\subseteq \mathfrak{P}(\Omega)$ and the $P$-a.s.\ order for some $P\in  \mathfrak{P}(\Omega)$ (typically $P\ll \mathcal{P}$). \\
Often we will, as is common practice, identify equivalence classes of random variables with their representatives. However, sometimes it will be helpful to distinguish between them to avoid confusion. Let us clarify this further: $X$ is an equivalence class of random variables if there exists an equivalence relation $ \sim$ on $\mathcal L^0$ such that $X = \{f\in \mathcal{L}^0 \colon f \sim g\}$ for some $g \in \mathcal{L}^0$. A measure $P\in \mathfrak{P}(\Omega)$ is consistent with the equivalence relation $\sim$ if
\begin{equation*}
    \forall f,g\in \mathcal{L}^0 \quad f \sim g \Longrightarrow P[f = g] = 1.
\end{equation*}
In that case we, for instance, write $E_P[X]$ for the expectation of $X$ under $P$, which actually means $E_P[f]$ for any $f\in X$ provided the latter integral is well-defined. Also, we will write expressions like $P[X = Y]$, where $Y$ is another equivalence class of random variables with respect to the same equivalence relation $\sim$, actually meaning $P[f = g]$ for arbitrary $f \in X$ and $g \in Y$. The difference here to the usual convention of identifying equivalence classes of random variables with their representatives is that the equivalence relation $\sim$ might not be given by $P$-a.s.\ equality, but $P$ is only assumed to be consistent with that equivalence relation in the above sense. A typical example is the equivalence relation given by $\mathcal P$-q.s.\ equality of random variables and $P\in\mathfrak{P}_c(\Omega)$.

\subsection{Supported Measures and Class (S) Robustness} \label{sec:support}

Supported measures $\mu \in ca_{c}$ play a key role in handling robustness. This concept is also known in statistics, see \cite{LMS2022} for a detailed review.

\begin{defi} \label{defi:supp}
    \begin{enumerate}[itemindent=25pt, leftmargin=0pt, nosep]
        \item A measure $\mu\in ca_{c+}$ is called supported if there is an event $S(\mu) \in \mathcal{F}$ such that
        \begin{enumerate}[itemindent=47pt, leftmargin=0pt, nosep]
            \item $\mu[S(\mu)^{c}] = 0$;
            
            \item whenever $N \in \mathcal{F}$ satisfies $\mu[N \cap S(\mu)] = 0$, then $N \cap S(\mu)$ is $\mathcal{P}$-polar.
        \end{enumerate}
        The set $S(\mu)$ is called the (order) support of $\mu$.
        
        \item A signed measure $\mu \in ca_{c}$ is supported if  $\lvert \mu \rvert$  is supported where
        \begin{equation*}
            \lvert \mu \rvert(A) := \sup \{\mu[B] - \mu[A \setminus B] \colon B \in \mathcal{F}, B \subseteq A\}, \quad  A\in \mathcal F,
        \end{equation*} is the total variation of $\mu$.
        
        \item We set
        \begin{equation*}
            sca_{c} := \{\mu \in ca_{c} \colon \mu \text{ is supported}\},
        \end{equation*}
        the space of all supported signed measures in $ca_{c}$, and $sca_{c+} := sca_{c} \cap ca_{c+}$.
    \end{enumerate}
\end{defi}

Note that if two sets $S, S' \in \mathcal{F}$ satisfy conditions (a) and (b) in Definition \ref{defi:supp}(1), then $\chi_{S} = \chi_{S'}$ $\mathcal{P}$-q.s. (${\bf 1}_{S}={\bf 1}_{S'}$), i.e., the symmetric difference $S \bigtriangleup S'$ is $\mathcal{P}$-polar. The order support $S(\mu)$ is thus usually not unique as an event, but only unique up to $\mathcal{P}$-polar events. In the following $S(\mu)$ therefore denotes an arbitrary version of the order support.

\begin{defi}
    A net $(X_{\alpha})_{\alpha \in I} \subset L^{0}_{c}$ is order convergent to $X \in L^{0}_{c}$, denoted $X_{\alpha} \order{c} X$,  if there is another net $(Y_{\alpha})_{\alpha \in I} \subset L^{0}_{c}$ with the same index set $I$ which is decreasing ($\alpha, \beta \in I$ and $\alpha \leq \beta$ imply $Y_{\beta} \preccurlyeq Y_{\alpha}$), satisfies $\inf_{\alpha \in I} Y_{\alpha} = 0$, and for all $\alpha \in I$ it holds that $ \lvert X_{\alpha} - X \rvert \preccurlyeq Y_{\alpha}$. Here, as usual, $\inf_{\alpha \in I} Y_{\alpha}$ denotes the largest lower bound of the net $(Y_{\alpha})_{\alpha \in I}$.
\end{defi}

Note that if $\mathcal P=\{P\}$, then $c=P$, and hence order convergence on $L^0_P$ with respect to the $P$-a.s.\ order is naturally denoted by $X_{\alpha} \order{P} X$. 

For $\mu\in ca_c$ consider the functional 
\begin{equation} \label{eq:sca:ordercont}
    L^\infty_c\ni X\mapsto \int X d\mu.
\end{equation}
This functional is called $\sigma$-order continuous if for every sequence $(X_n)_{n\in \N}\subset L^\infty_{c}$ such that $X_n\order{c} X\in L^\infty_c$ we have $\int X_n d\mu \to \int X d\mu$, and order continuous if for every net $(X_\alpha)_{\alpha\in I}\subset L^\infty_{c}$ such that $X_\alpha\order{c} X\in L^\infty_c$ we have $\int X_\alpha d\mu \to \int X d\mu$.
Indeed, the space of order continuous linear functionals may be identified with $sca_c$ via \eqref{eq:sca:ordercont}. In the same way $ca_c$ is identified with the space of all $\sigma$-order continuous functionals, and any $\mu\in ca_c\setminus sca_c$ induces a linear $\sigma$-order continuous functional which is not order continuous. Note that in robust frameworks $ca_c\setminus sca_c\neq \emptyset$ is often the case. We refer to \cite{LMS2022} for the latter facts, and, in general, for a concise but comprehensive discussion of the spaces $ca_c$ and $sca_c$.

Stochastic models, for instance of financial markets, which do not assume a dominating probability measure are often referred to as being robust; see \cite{BN2015, BFM2017, N2014, STZ2011} and the references therein. In \cite{LMS2022} an important subclass of such robust models, namely the models of class (S) defined next, are discussed:
\begin{defi}
    $\mathcal{P}$ is of class (S) if there exists a set of supported probability measures $\mathcal{Q}$ (i.e.\ $\mathcal{Q}\subseteq \mathfrak{P}_c(\Omega)\cap sca_{c} $) such that $\mathcal{Q} \approx \mathcal{P}$. In that case we call $\mathcal{Q}$ a supported alternative of $\mathcal{P}$.
\end{defi}

Let us briefly comment on the significance of the class (S) property: Suppose that $\mathcal P$ is of class (S) and let $\mathcal{Q}$ be a supported alternative of $\mathcal{P}$. As $\mathcal{Q}\approx \mathcal{P}$, the $\mathcal{Q}$-q.s.\ order coincides with the $\mathcal{P}$-q.s.\ order $\preccurlyeq$. Hence, when arguing by means of the $\mathcal{P}$-q.s.\ order---think of robust superhedging, for instance---which means to prove some statement for each $P\in \mathcal{P}$, we may indeed switch to $\mathcal{Q}$ and prove the corresponding statement for each $Q\in \mathcal{Q}$. Here we often benefit from $Q$ being supported. In \cite{LMS2022} it is shown how the class (S) property is important, and indeed necessary, in many situations to handle robustness in non-dominated frameworks. 

\begin{defi}
    Let $\mathcal{Q} \subseteq \mathfrak{P}_c(\Omega) \cap sca_{c}$. We say that $\mathcal{Q}$ has disjoint supports if, for all $Q, Q' \in \mathcal{Q}$ such that $Q \neq Q'$,  $\mathbf{1}_{S(Q)} \wedge \mathbf{1}_{S(Q')} = 0$, that is $S(Q)\cap S(Q')$ is a $\mathcal{P}$-polar event.
\end{defi}

\begin{lem}[{see \cite[Lemma 3.7]{LMS2022}}]\label{lem:disjoint:alt}
    Suppose $\mathcal{P}$ is of class (S). Then there exists a supported alternative $\mathcal{Q} \approx \mathcal{P}$ with disjoint supports. $\mathcal{Q}$ will be referred to as a disjoint supported alternative.
\end{lem}

The following Example~\ref{exmp:classS} serves as a simple illustration of the preceding discussion. In Example~\ref{exmp:lit} below, we collect prominent examples of class (S) models in the literature. The class (S) property of the latter examples is extensively discussed in \cite{LMS2022}.  

\begin{exmp} \label{exmp:classS}
    Consider the unit interval $\Omega = [0, 1]$ equipped with the Borel-$\sigma$-algebra $\mathcal{F} = \mathcal{B}(\Omega)$. Let $\mathcal{P} := \mathfrak{P}(\Omega)$, and let $\mathcal{Q} := \{\delta_{\omega} \colon \omega \in \Omega\}$ be the set of all Dirac probability measures. Clearly, $\mathcal{Q} \approx \mathcal{P}$. In particular, the  upper probability $c$ induced by $\mathcal{P}$ satisfies
    \begin{equation*}
        c[A] = \sup_{P\in \mathcal{P}} P[A]= \sup_{\omega\in \Omega}\delta_\omega[A] =
        \begin{cases}
            1, &\text{ if } A \neq \emptyset, \\
            0, &\text{ if } A = \emptyset,
        \end{cases}
        \quad A\in \mathcal{F}.
    \end{equation*}
    Thus, the only $\mathcal{P}$-polar set is $\emptyset$. For each $\omega \in \Omega$, let $S(\delta_{\omega}) := \{\omega\}$. We have:  
    \begin{enumerate}[itemindent=25pt, leftmargin=0pt, nosep]
        \item[(a)] $\delta_{\omega}[\Omega \setminus \{\omega\}] = 0$.

        \item[(b)] Let $N \in \mathcal{F}$ such that $\delta_{\omega}[N \cap S(\delta_{\omega})] = \delta_{\omega}[N \cap \{\omega\}] = 0$. Then, $\omega \not \in N$ and therefore $N \cap \{\omega\}=\emptyset$ is $\mathcal{P}$-polar. 
    \end{enumerate}
    Hence, for each $\omega \in \Omega$, $\delta_{\omega}$ is supported with support $S(\delta_{\omega}) = \{\omega\}$. Moreover, those supports are obviously disjoint. In sum, we have shown that $\mathcal{Q}$ is a disjoint supported alternative to $\mathcal{P}$. In particular, $\mathcal{P}$ is of class (S). \\
    Also note that in this case $ca_{c} \neq sca_{c}$. Indeed consider, for instance, the Lebesgue measure $\lambda$ on $(\Omega,\mathcal{F})$. For any $S \in \mathcal{F}$ with $\lambda[S^{c}] = 0$, and for any non-empty countable subset $N$ of $S$, we have $\lambda[N \cap S] = \lambda[N] = 0$, but $N \cap S = N$ is not $\mathcal{P}$-polar. Hence, $\lambda$ is not supported, that is $\lambda\in ca_{c}\setminus sca_{c}$.
\end{exmp}

\begin{exmp} \label{exmp:lit}
    The underlying robust probabilistic models of the following financial models are all of class (S).  We refer to \cite[Section 3.2]{LMS2022} for the details and in particular the proofs of the class (S) property. 
    \begin{enumerate}[itemindent=25pt, leftmargin=0pt, nosep]
        \item The financial models on product spaces given in \cite{C2024} and \cite{CFR2022}, see \cite[Section~3.2.1]{LMS2022}.
        
        \item The volatility uncertainty models discussed in \cite{C2012} and \cite{STZ2011}, see \cite[Section~3.2.2]{LMS2022}.
        
        \item A model of innovation considered in \cite{AGP2017}, see \cite[Section~3.2.3]{LMS2022}.
        
        \item The models applied to study the superhedging problem in \cite{BKN2020}, \cite{BFH2019}, \cite{HO2018}, and \cite{M2003}, see \cite[Section~3.2.4]{LMS2022}.
        
        \item The robust binomial model considered in \cite{BC2020}, see \cite[Section~3.2.5]{LMS2022}.
    \end{enumerate}
\end{exmp}

\subsection{$\mathcal{P}$-sensitive Sets}

Let $\mathcal{P} \subseteq \mathfrak{P}(\Omega)$. A property that will play a major role in our studies is the so-called $\mathcal{P}$-sensitivity of subsets of $L^0_c$ defined in the following, see also \cite{MMS2018}. To this end, recall that $[f]_c$ denotes the equivalence class in $L^0_c$ generated by $f\in \mathcal{L}^0$, whereas $[f]_Q$ is the equivalence class generated by $f$ in $L^0_Q$, that is under $Q$-a.s.\ equality. The following map identifies any $X,Y\in L^0_c$ which appear to coincide under $Q$, that is $Q[f = g] = 1$ for $f \in X$ and $g \in Y$:
\begin{equation*}
    j_{Q} \colon L^{0}_{c} \rightarrow L^{0}_{Q}, \quad [f]_c \mapsto [f]_Q.
\end{equation*}

\begin{defi}\label{def:sensitivity}
    A set $\mathcal{C} \subseteq L^0_c$ is called $\mathcal{P}$-sensitive if
    \begin{equation*}
        \mathcal{C} = \bigcap_{Q \in \mathfrak{P}_c(\Omega)} j_{Q}^{-1} \circ j_{Q}(\mathcal{C}).
    \end{equation*}
\end{defi}

$\mathcal{P}$-sensitivity means that the set $\mathcal C$ is completely determined by its image under each model $Q\in \mathfrak{P}_c(\Omega)$, so if $X\in L^0_c$ looks like a member of $\mathcal{C}$ under each $Q\in \mathfrak{P}_c(\Omega)$ (i.e.\ $j_{Q}(X) \in j_{Q}(\mathcal{C})$ for all $Q \in  \mathfrak{P}_c(\Omega)$), then in fact $X \in \mathcal{C}$. Note that always $ \mathcal{C} \subseteq \bigcap_{Q \in \mathfrak{P}_c(\Omega)} j_{Q}^{-1} \circ j_{Q}(\mathcal{C})$, so the nontrivial inclusion is $ \bigcap_{Q \in \mathfrak{P}_c(\Omega)} j_{Q}^{-1} \circ j_{Q}(\mathcal{C})\subseteq \mathcal C$. Trivially, if $\mathcal{P}=\{P\}$, then every set $\mathcal{C}\subseteq L^0_P$ is $P$-sensitive. 

It will sometimes turn out to be useful to know a stronger sensitive representation of $\mathcal{C}$:
 
\begin{defi}
    Let $\mathcal{C} \subseteq L^0_c$. $\mathcal{Q} \subseteq \mathfrak{P}_c(\Omega)$ is called a reduction set for $\mathcal{C}$ if $\mathcal{Q}\neq \emptyset$ and  \begin{equation}\label{eq:reduction}
        \mathcal{C} = \bigcap_{Q \in\mathcal{Q}} j_{Q}^{-1} \circ j_{Q}(\mathcal{C}).
    \end{equation}
\end{defi}
 
By definition, any $\mathcal{P}$-sensitive set admits the reduction set $\mathfrak{P}_c(\Omega)$. The following lemma relates reduction sets to each other and in particular shows that any set satisfying \eqref{eq:reduction} is indeed $\mathcal{P}$-sensitive.  

\begin{lem}
    Let $\mathcal{C}\subseteq L^0_c$. 
    \begin{enumerate}[itemindent=25pt, leftmargin=0pt, nosep]
        \item Consider a reduction set $\mathcal{Q}_1$ for $\mathcal{C}$ and any other set of probability measures $\mathcal{Q}_{2} \subseteq \mathfrak{P}_c(\Omega)$ such that $\mathcal{Q}_{1} \subseteq \mathcal{Q}_{2}$. Then $\mathcal{Q}_{2}$ is also a reduction set for $\mathcal C$. 
        
        \item If $\mathcal C$ satisfies \eqref{eq:reduction} for some non-empty set $\mathcal{Q} \subseteq \mathfrak{P}_c(\Omega)$, then $\mathcal C$ is $\mathcal{P}$-sensitive.

        \item If $\mathcal{C}$ is $\mathcal{P}$-sensitive and $\tilde{\mathcal P}\subseteq \mathfrak{P}(\Omega)$ dominates $\mathcal P$, i.e., $\mathcal P \ll \tilde{\mathcal P}$, then $\mathcal C$ is $\tilde{\mathcal P}$-sensitive, where $\tilde{\mathcal P}$-sensitive means that \begin{equation*}
        \mathcal{C} = \bigcap_{\{P\in \mathfrak{P}(\Omega)\mid P \ll \tilde{\mathcal P}\}} j_{P}^{-1} \circ j_{P}(\mathcal{C}).
    \end{equation*}
    \end{enumerate}
\end{lem}
\begin{proof} The first statement follows from
    \begin{equation}\label{eq:help:sen:1}
        \mathcal{C} \subseteq \bigcap_{Q \in \mathcal{Q}_{2}} j_{Q}^{-1} \circ j_{Q}(\mathcal{C}) \subseteq \bigcap_{Q \in \mathcal{Q}_{1}} j_{Q}^{-1} \circ j_{Q}(\mathcal{C}) = \mathcal{C}.
    \end{equation}
    The second assertion follows from 1.\ by choosing $\mathcal{Q}_1=\mathcal{Q}$ and $\mathcal{Q}_2=\mathfrak{P}_c(\Omega)$. Finally,  $\mathcal P \ll \tilde{\mathcal P}$ implies that $\mathfrak{P}_c(\Omega) \subseteq \{P\in \mathfrak{P}(\Omega) \colon P \ll \tilde{\mathcal P}\}$, so we may argue as in \eqref{eq:help:sen:1}. 
\end{proof}

The reason for considering other reduction sets than simply $\mathfrak{P}_c(\Omega)$ will become evident throughout the paper. As we will see next, $\mathcal{P}$-sensitive sets are stable under intersection.

\begin{lem}\label{lem:intersect:Psensi}
    Let $I$ be a non-empty index set and let $\mathcal{C}_{\alpha} \subseteq L^{0}_{c}$, $\alpha\in I$, be $\mathcal{P}$-sensitive. Then
    \begin{equation*}
        \mathcal{C} := \bigcap_{\alpha \in I} \mathcal{C}_{\alpha}
    \end{equation*}
    is also $\mathcal{P}$-sensitive. If $\mathcal{Q}_\alpha \subseteq \mathfrak{P}_c(\Omega)$ is a reduction set for $\mathcal{C}_{\alpha}$ for each $\alpha\in I$, then $\mathcal Q := \bigcup_{\alpha\in I} \mathcal{Q}_\alpha$ is a reduction set for $\mathcal C$. 
\end{lem}
\begin{proof}
    Suppose that $j_Q(X)\in j_Q(\mathcal C)$ for all $Q\in \mathcal Q$. Then in particular $j_Q(X)\in j_Q(\mathcal C)$ for all $Q\in \mathcal Q_\alpha$ and all $\alpha\in I$. Hence, $X\in \mathcal{C}_\alpha$ for all $\alpha\in I$.
\end{proof}

In the following Example~\ref{exmp:pSensi}, we present first simple examples of $\mathcal{P}$-sensitive sets. 
For a more detailed discussion of $\mathcal{P}$-sensitivity, and further examples, we refer to Section~\ref{sec:pSensi:rel}.

\begin{exmp} \label{exmp:pSensi}
For $a\in \R$ consider the sets
    \begin{equation*}
        \mathcal{C}_-^a := \{X \in L^{0}_{c} \colon X \preccurlyeq a\} \quad \text{and} \quad \mathcal{C}_+^a := \{X \in L^{0}_{c} \colon a \preccurlyeq X\}.
    \end{equation*}
Let us show that $\mathcal{C}_-^a$ and $\mathcal{C}_+^a$ are $\mathcal{P}$-sensitive, and that $\mathcal{P}$ serves as a reduction set in both cases. To this end, fix $X \in L^{0}_{c}$ such that $j_{P}(X) \in j_{P}(\mathcal{C}_-^a)$ for every $P \in \mathcal{P}$. Then, for each $P \in \mathcal{P}$, there exists $Y^{P} \in \mathcal{C}_-^a$ such that $j_{P}(X) = j_{P}(Y^{P})$, meaning that $P[X = Y^{P}] = 1$. Hence, $P[X \leq a] = P[Y^{P} \leq a] = 1$ for every $P \in \mathcal{P}$. Therefore, $X \in \mathcal{C}_-^a$. We have thus shown that  $\mathcal{C}_-^a$ is $\mathcal{P}$-sensitive with  reduction set $\mathcal{P}$. The same reasoning proves $\mathcal{P}$-sensitivity of $\mathcal{C}_+^a$. 
\end{exmp}

\section{Bipolar Representations} \label{sec:bip:rep}

We recall the well-known bipolar theorem on $L^0_{P+}$ given in \cite[Theorem 1.3]{BS1999} and used in the seminal study by \cite{KS1999} of the utility maximization problem: 

\begin{thm} \label{thm:BS}
    Let $P\in \mathfrak{P}(\Omega)$ and $\mathcal{C}\subseteq L^0_{P+}$ be non-empty. Define the polar of $\mathcal{C}$ as
    \begin{equation*}
        \mathcal{C}^\circ := \{Y \in L^0_{P+} \colon \forall X \in \mathcal{C} \ E_P[XY] \leq 1\}.
    \end{equation*}
    Then $\mathcal{C}^\circ$ is a non-empty, $P$-closed, convex, and solid subset of $L^0_{P+}$, and the bipolar
    \begin{equation} \label{eq:BSrep}
        \mathcal{C}^{\circ\circ} := \{X \in L^0_{P+} \colon \forall Y \in \mathcal{C}^\circ \ E_P[XY] \leq 1\}
    \end{equation}
    of $\mathcal C$ is the smallest $P$-closed, convex, solid set in $L^0_{c+}$ containing $\mathcal C$. In particular if $\mathcal C$ is $P$-closed, convex, and solid, then $\mathcal{C} = \mathcal{C}^{\circ\circ}$.   
\end{thm}

$P$-closedness in Theorem~\ref{thm:BS} means that the respective set is closed under convergence in probability with respect to $P$. The definition of solidness is recalled next: 

\begin{defi}
    Let $\mathcal{C} \subseteq L^{0}_{c}$. $\mathcal{C}$ is called solid in $L^{0}_{c}$ if $X \in \mathcal{C}$, $Y\in L^0_c$ and $|Y| \preccurlyeq |X|$ imply $Y \in \mathcal{C}$. $\mathcal{C}$ is solid in $L^{0}_{c+}$ if $\mathcal{C} \subseteq L^{0}_{c+}$ and $X \in \mathcal{C}$, $Y\in L^0_{c+}$ and $Y \preccurlyeq X$ imply $Y \in \mathcal{C}$. We simply say that $\mathcal C$ is solid, if $\mathcal C$ is either solid in $L^0_c$ or solid in $L^0_{c+}$.
\end{defi}

Note that our usage of the term 'solid' should not cause confusion, because, apart from the empty set, a set which is solid in $L^0_{c+}$ cannot be solid in $L^0_c$ and vice versa.  In fact, a non-empty solid set in $L^0_{c+}$ only comprises non-negative elements while a non-empty solid set $\mathcal{C}$ in $L^0_c$ is symmetric in the sense that $X\in \mathcal{C}$ implies $-X\in \mathcal{C}$ (even $YX\in \mathcal{C}$ for all $Y\in L^0_c$ which take values in $\{-1,1\}$). Indeed, one verifies that  the intersection of a solid set in $L^0_c$ with the positive cone $L^0_{c+}$ is a solid set in $L^0_{c+}$.

In Theorem~\ref{thm:BS} we have $\mathcal{P}=\{P\}$, and the subset $\mathcal{C} \subseteq L^{0}_{P+}$ is solid if and only if $X \in \mathcal{C}$, $Y\in L^0_{P+}$, and $Y \leq_P X$  imply $Y \in \mathcal{C}$.

We also like to mention a useful strengthening of Theorem~\ref{thm:BS}, still with ambient space $L^0_{P+}$, given in \cite[Corollary 2.7]{KS2011}:

\begin{thm} \label{thm:KS}
    Let $P\in \mathfrak{P}(\Omega)$ and $\mathcal{C}\subseteq L^0_{P+}$ be non-empty. Define the polar of $\mathcal{C}$ as
    \begin{equation*}
        \mathcal{C}^\circ := \{Y \in L^\infty_{P+} \colon \forall X \in \mathcal{C} \ E_P[X Y] \leq 1\}.
    \end{equation*}
    Then $\mathcal{C}^\circ$ is a non-empty, $\sigma(L^\infty_{P},L^\infty_{P})$-closed, convex, solid subset of $L^\infty_{P}$, and the bipolar
    \begin{equation} \label{eq:KSrep}
        \mathcal{C}^{\circ\circ} := \{X \in L^0_{P+} \colon \forall Y \in \mathcal{C}^\circ \ E_P[X Y] \leq 1\}
    \end{equation}
    of $\mathcal C$ is the smallest $P$-closed, convex, solid set in $L^0_{c+}$ containing $\mathcal C$. In particular if $\mathcal C$ is $P$-closed, convex, and solid, then $\mathcal{C}=\mathcal{C}^{\circ\circ}$.  
\end{thm}

The important difference between Theorems~\ref{thm:BS} and \ref{thm:KS} is that the latter replaces the dual cone $L^0_{P+}$ of Theorem~\ref{thm:BS} by $L^\infty_{P+}$. The boundedness of elements in $L^\infty_{P+}$ will prove helpful when deriving robust bipolar theorems on $L^0_{c+}$ by {\em lifting} those on $L^0_{P+}$ for $P\in \mathcal P$, see Section~\ref{sec:bip:thm}. Note that by solidness and monotone convergence one directly verifies that the sets in \eqref{eq:BSrep} and \eqref{eq:KSrep} indeed coincide.

\subsection{A Reverse Perspective}

In this section we collect some simple observations on necessary conditions for a bipolar representation which will, however, set the direction of our further studies.

\begin{prop} \label{prop:pSensi}
    Let $\mathcal{X}\subseteq L^0_c$ be a non-empty convex subset and suppose that the non-empty set $\mathcal C\subseteq \mathcal{X}$ admits a representation
    \begin{equation} \label{eq:bipolar:1}
        \mathcal{C} = \{X \in \mathcal{X} \colon \forall h \in \mathcal{K} \ h(X) \leq 1\}
    \end{equation}
    where $\mathcal{K}$ denotes a non-empty set of functions $h:\mathcal{X}\to \R\cup\{-\infty,\infty\}$. 
    \begin{enumerate}[itemindent=25pt, leftmargin=0pt, nosep]
        \item If each $h\in \mathcal{K}$ is dominated by a probability measure $Q\in \mathfrak{P}_c(\Omega)$ in the sense that
        \begin{equation*}
            \forall X, Y \in \mathcal{X} \quad Q[X = Y] = 1 \Longrightarrow h(X) = h(Y),
        \end{equation*}
        then $\mathcal{C}$ is $\mathcal P$-sensitive. Any set $\mathcal Q\subseteq \mathfrak{P}_c(\Omega)$ such that every $h\in \mathcal K$ is dominated by some $Q\in \mathcal Q$ serves as reduction set for $\mathcal C$.
        
        \item If the functions $h$ are convex, then $\mathcal{C}$ is necessarily convex.
        
        \item If the functions $h$ are monotone with respect to some partial order $\triangleleft$ on $\mathcal X$, i.e., for all $X, Y\in \mathcal{X}$ we have that $X\triangleleft Y$ implies $h(X)\leq h(Y)$, then $\mathcal{C}$ is monotone with respect to $\triangleleft$, i.e., $Y\in \mathcal C$, $X\in \mathcal{X}$ and $X\triangleleft Y$ imply $X\in \mathcal C$.
        
        \item If the functions $h$ are (sequentially) lower semi-continuous with respect to some topology $\tau$ on $\mathcal X$, then $\mathcal C$ is necessarily (sequentially) $\tau$-closed.
    \end{enumerate}  
\end{prop}
\begin{proof}
    1. Let $\mathcal Q\subseteq \mathfrak{P}_c(\Omega)$ be such that every $h\in \mathcal K$ is dominated by some $Q\in \mathcal Q$. We have to prove that if $X\in \mathcal X$ satisfies $j_Q(X)\in j_Q(\mathcal{C})$ for all $Q\in \mathcal Q$, then $X\in \mathcal C$. To this end, fix such an $X$ and let $h\in \mathcal{K}$ be arbitrary and choose $Q\in \mathcal Q$ which dominates $h$. There is $Y\in \mathcal {C}$ such that $j_Q(Y)=j_Q(X)\in j_Q(\mathcal{C})$. As $Q[X = Y] = 1$, we obtain
    \begin{equation*}
        h(X) = h(Y) \leq 1.
    \end{equation*}
    Since $h\in \mathcal{K}$ was arbitrarily chosen, we conclude that $X\in \mathcal{C}$.
    
    2., 3. and 4. are easily verified. 
\end{proof}

As our focus lies on bipolar representations for subsets of $\mathcal{X}=L^0_{c+}$, let us further refine the implications of Proposition~\ref{prop:pSensi} in that setting. If  $\mathcal{X}=L^0_{c+}$ it seems natural that the functions $h$ appearing in the representation \eqref{eq:bipolar:1} are of type $h(X)=E_P[XZ]$ for some $P\in \mathfrak{P}(\Omega)$ and $Z\in L^0_{c+}$. Under this assumption the following Corollary~\ref{cor:l0:bipolar} provides more  information. However, before we are able to state the corollary we need to introduce some further notation: Let $X_{n}$, ${n \in \N}$, and $X$ be equivalence classes of random variables with respect to the same equivalence relation on $\mathcal{L}^0$, and let $P\in \mathfrak{P}(\Omega)$ be consistent with that equivalence relation, see Section~\ref{sec:basics}.
We will write $X_n\overset{P}{\longrightarrow} X$ to indicate that $(X_{n})_{n \in \N}$ converges to $X$ in probability with respect to $P$, that is for any choice $f_n\in X_n$ and $f\in X$ the sequence of random variables $(f_{n})_{n \in \N}$ converges to $f$ in probability with respect to $P$. For a subset $\mathcal Q$ of $ \mathfrak{P}(\Omega)$ we write $X_n\overset{\mathcal Q}{\longrightarrow} X$  to indicate that every $Q\in \mathcal Q$ is consistent with the equivalence relation defining $X_n$, $n\in \N$, and $X$, and $X_n\overset{Q}{\longrightarrow} X$  for all $Q\in \mathcal Q$. 

\begin{defi}
    Let $\mathcal Q\subseteq \mathfrak{P}(\Omega)$ be non-empty.
    A set $\mathcal{C}\subseteq L^0_c$  is called $\mathcal Q$-closed if $(X_{n})_{n \in \N} \subseteq \mathcal{C}$ and  $X_n\overset{\mathcal Q}{\longrightarrow} X$ for some $X\in L^0_{c}$ implies that $X\in \mathcal C$.
\end{defi}

Note that if $\tilde{ \mathcal{Q}}\subseteq \mathcal{Q}\subseteq \mathfrak{P}(\Omega)$ and if $\mathcal C$ is $\tilde{ \mathcal{Q}}$-closed, then $\mathcal C$ is also $\mathcal Q$-closed. In particular, any $\mathcal{Q}$-closed set is $\mathfrak{P}(\Omega)$-closed, and if $\mathcal{Q}\subseteq \mathfrak{P}_c(\Omega)$, any $\mathcal{Q}$-closed set is $\mathfrak{P}_c(\Omega)$-closed.

\begin{cor} \label{cor:l0:bipolar}
    Suppose that the non-empty set $\mathcal{C} \subseteq L^0_{c+}$ admits a bipolar representation of the form
    \begin{equation*}
        \mathcal{C} = \{X \in L^{0}_{c+} \colon \forall (P, Z) \in \mathcal{K} \ E_{P}[Z X] \leq 1\}
    \end{equation*}
    where $\mathcal{K} \subseteq \mathfrak{P}_c(\Omega)\times L^{0}_{c+}$ is non-empty. Then $\mathcal{C}$ is $\mathcal P$-sensitive, convex, solid, and $\mathfrak{P}_c(\Omega)$-closed.  Let $\mathcal Q \subseteq \mathfrak{P}_c(\Omega)$ denote any set of probabilities such that for all $(P, Z) \in \mathcal{K}$ there is $Q\in \mathcal Q$ with $P\ll Q$. Then $\mathcal Q$ serves as reduction set for $\mathcal{C}$ and  $\mathcal{C}$ is in fact $\mathcal{Q}$-closed.
\end{cor}
\begin{proof}
    Convexity, solidness,  and $\mathcal{P}$-sensitivity with reduction set $\mathcal{Q}$ immediately follow from Proposition~\ref{prop:pSensi}. Also $\mathcal{Q}$-closedness is a consequence of  Proposition~\ref{prop:pSensi} since for any $(P, Z) \in \mathcal{K}$ the function $X\ni L^0_{c+}\mapsto E_{P}[Z X]$ is sequentially lower semi-continuous with respect to $\mathcal{Q}$-convergence. Indeed, consider any $r\in \R$ and let  $(X_{n})_{n \in \N} \subseteq L^0_{c+}$ and $X \in L^{0}_{c+}$ such that $X_n\overset{\mathcal Q}{\longrightarrow} X$ and $E_{P}[Z X_{n}]\leq r$ for all $n\in \N$.  As $P\ll Q$ for some $Q\in \mathcal Q$ and $X_n\overset{Q}{\longrightarrow} X$,  there is a subsequence $(X_{n_k})_{k\in \N}$ of $(X_{n})_{n \in \N}$ converging $Q$-a.s.\ and thus $P$-a.s.\ to $X$.
    Hence, by Fatou's lemma 
    $$E_{P}[Z X] \leq  \liminf_{k \rightarrow \infty} E_{P}[Z X_{n_k}]\leq r.$$ 
\end{proof}

Note the relation between the reduction set and the closedness of $\mathcal C$ stated in Corollary~\ref{cor:l0:bipolar}.

\subsection{Lifting Bipolar Representations}\label{sec:lifting}

As we have seen above, $\mathcal{P}$-sensitivity is necessary for a bipolar representation. In this section we will see how $\mathcal{P}$-sensitivity can be used to obtain a robust bipolar representation by lifting known bipolar theorems in dominated frameworks to the robust model $L^0_c$.

Throughout this section let $\mathcal X$ be a convex subset of $L^0_c$, and let $\mathcal{C}\subseteq \mathcal{X}$ be a non-empty $\mathcal{P}$-sensitive set with reduction set $\mathcal{Q} \subseteq \mathfrak{P}_c(\Omega)$. Further, let $\mathcal{X}_Q:=j_Q(\mathcal{X})$ and $\mathcal{C}_Q:=j_Q(\mathcal{C})$ for all $Q\in \mathcal{Q}$. For each $Q\in \mathcal{Q}$, we denote by $\mathcal{Y}_Q$ a non-empty set of mappings $l:\mathcal{X}_Q\to \R\cup\{-\infty,\infty\}$ and let
\begin{equation*}
    \mathcal{C}^\circ_Q := \{l \in \mathcal{Y}_Q \colon \forall X \in \mathcal{C}_Q \ l(X) \leq 1\}
\end{equation*}
and
\begin{equation*}
    \mathcal{C}^{\circ\circ}_Q := \{X \in \mathcal{X}_Q \colon \forall l \in \mathcal{C}^\circ_Q \ l(X) \leq 1\}.
\end{equation*}
Set
\begin{equation*}
    \mathcal{C}^\circ:=\bigcup_{Q\in \mathcal{Q}}\{l\circ j_Q \colon l \in \mathcal{C}^\circ_Q\}
\end{equation*}
and
\begin{equation*}
    \mathcal{C}^{\circ\circ} := \{X \in \mathcal{X} \colon \forall h \in \mathcal{C}^\circ \ h(X) \leq 1\}.
\end{equation*}

\begin{thm} \label{thm:bipolar:ext}
    Suppose that $\mathcal{C}_Q=\mathcal{C}^{\circ\circ}_Q$ for all $Q\in \mathcal{Q}$. Then $\mathcal{C}=\mathcal{C}^{\circ\circ}$.
\end{thm}
\begin{proof}
    Let $X\in \mathcal{C}$, then $j_Q(X)\in \mathcal{C}_Q$ and thus $l(j_Q(X))\leq 1$ for all $l\in \mathcal{C}_Q^\circ$ and $Q\in \mathcal{Q}$. Hence, $X\in \mathcal{C}^{\circ\circ}$. 
    Now let $X\in \mathcal{C}^{\circ\circ}$. Then for any $Q\in \mathcal{Q}$ we have that $l(j_Q(X))\leq 1$ for all $l\in \mathcal{C}_Q^\circ$, that is $j_Q(X)\in \mathcal{C}^{\circ\circ}_Q$. Since, by assumption, $\mathcal{C}_Q=\mathcal{C}^{\circ\circ}_Q$ for all $Q \in \mathcal{Q}$, we obtain $j_Q(X)\in \mathcal{C}_Q$ for all $Q \in \mathcal{Q}$. As $\mathcal{Q}$ is a reduction set for $\mathcal{C}$, we conclude that $X\in \mathcal C$.
\end{proof}

Clearly, the assumption of Theorem~\ref{thm:bipolar:ext} that $\mathcal{C}_{Q} = \mathcal{C}_{Q}^{\circ\circ}$ holds for all $Q \in \mathcal{Q}$ is rather abstract and does not provide a good bipolar theorem at first sight. However, we will use Theorems~\ref{thm:BS} and \ref{thm:KS} to conclude that under some conditions on $\mathcal C\subseteq L^0_{c+}$, each $\mathcal{C}_{Q}\subseteq L^0_{Q+}$ admits a bipolar representation $\mathcal{C}_{Q} = \mathcal{C}_{Q}^{\circ\circ}$. Then Theorem~\ref{thm:bipolar:ext} allows to lift this bipolar representation to $L^0_{c+}$. The required conditions on $\mathcal C$ will, of course, comprise convexity and solidness with respect to the $\mathcal P$-quasi-sure order (Corollary~\ref{cor:l0:bipolar}), and  we also need to discuss reasonable closure properties (again Corollary~\ref{cor:l0:bipolar}). The latter is the purpose of the next section.  

\section{Concepts of Closedness under Uncertainty} \label{sec:clsd}

Recall the discussion from Section~\ref{sec:lifting}. If we want to apply  Theorem~\ref{thm:BS} or \ref{thm:KS}, we need to ensure that every $j_Q(\mathcal{C})$ is $Q$-closed. A straightforward way of achieving this is to assume that $\mathcal{C}$ is $Q$-closed for each $Q \in \mathcal{Q}$. Yet, a still sufficient and indeed also necessary property is 
the following weaker requirement:

\begin{defi}
    Let $\mathcal C \subseteq L^0_c$ and $Q\in \mathfrak{P}_c(\Omega)$. $\mathcal{C}$ is called locally $Q$-closed if for each sequence $(X_{n})_{n \in \N} \subseteq \mathcal{C}$ and $X \in L^{0}_{c}$ such that $X_{n} \overset{Q}{\longrightarrow} X$ there exists $Y \in \mathcal{C}$ such that $j_{Q}(X) = j_{Q}(Y)$.
\end{defi}

\begin{lem} \label{lem:loc:clsd}
    Let $\mathcal{C} \subseteq L^{0}_{c}$ and $Q \in \mathfrak{P}_c(\Omega)$. $\mathcal{C}$ is locally $Q$-closed if and only if  $j_Q(\mathcal{C})$ is $Q$-closed.
\end{lem}
\begin{proof} We may assume that $\mathcal C\neq \emptyset$.
    Suppose that $\mathcal{C}$ is locally $Q$-closed. 
    Let $(X^{Q}_{n})_{n \in \N} \subseteq j_Q(\mathcal{C})$ and $X^Q\in L^0_{Q}$ such that
    $X^{Q}_{n} \overset{Q}{\longrightarrow} X^{Q}$.
        Pick $(X_{n})_{n \in \N} \subseteq \mathcal{C}$ such that $j_{Q}(X_{n})=X^{Q}_{n}$ and $X \in L^0_{c}$ such that $j_Q(X) = X^{Q}$. It follows that $X_{n} \overset{Q}{\longrightarrow} X$.
    As $\mathcal{C}$ is locally $Q$-closed, there exists $Y \in \mathcal{C}$ such that $ j_Q(\mathcal{C})\ni j_{Q}(Y)=  j_{Q}(X) = X^{Q}$. Thus, $\mathcal{C}_{Q}$ is $Q$-closed. \\
    Conversely, if $j_Q(\mathcal{C})$ is $Q$-closed and $(X_{n})_{n \in \N} \subseteq \mathcal{C}$ and $X\in L^0_{Q}$ such that
    $X_{n} \overset{Q}{\longrightarrow} X$, then $j_Q(X_{n}) \overset{Q}{\longrightarrow} j_Q(X)$ in $L^0_Q$ and thus $j_Q(X)\in j_Q(\mathcal{C})$. Now let $Y\in \mathcal C$ such that $j_Q(Y)=j_Q(X)$.
\end{proof}

So far we have encountered two concepts of closedness based on a reduction set $\mathcal{Q}$ for $\mathcal{C}$: $\mathcal{Q}$-closedness appeared as a necessary condition in Corollary~\ref{cor:l0:bipolar}, whereas local $Q$-closedness for all $Q\in \mathcal{Q}$  enables a lifting of Theorems~\ref{thm:BS} and \ref{thm:KS}. Interestingly, both notions are equivalent for $\mathcal{P}$-sensitive and solid sets:

\begin{prop} \label{prop:clsd:allQ:loc}
    Suppose that $\mathcal{C} \subseteq L^{0}_{c}$ is $\mathcal{P}$-sensitive with reduction set $\mathcal{Q} \subseteq \mathfrak{P}_c(\Omega)$. If $\mathcal C$ is locally $Q$-closed for all $Q \in \mathcal{Q}$, then $\mathcal C$ is $\mathcal Q$-closed. If additionally $\mathcal{C}$ is solid, then $\mathcal C$ is locally $Q$-closed for all $Q \in \mathcal{Q}$ if and only if $\mathcal C$ is $\mathcal Q$-closed.
\end{prop}
\begin{proof} Assume that $\mathcal C\neq \emptyset$.
    Suppose $\mathcal{C}$ is locally $Q$-closed for each $Q \in \mathcal{Q}$. Let $(X_{n})_{n \in \N} \subseteq \mathcal{C}$ and $X \in L^{0}_{c}$ such that $X_{n} \overset{\mathcal Q}{\longrightarrow} X$. By local $Q$-closedness, for each $Q \in \mathcal{Q}$, there exists $Y_{Q} \in \mathcal{C}$  such that $j_{Q}(X) = j_{Q}(Y_{Q})\in j_Q(\mathcal{C})$. Since $\mathcal Q$ is a reduction set for $\mathcal C$ we obtain $X \in \mathcal{C}$. Hence, $\mathcal{C}$ is $\mathcal Q$-closed. \\
    Now suppose that $\mathcal{C}$ is also solid and let $\mathcal{C}$ be $\mathcal{Q}$-closed. Fix $Q \in \mathcal{Q}$ and let $(X_{n})_{n \in \N} \subseteq \mathcal{C}$ such that $X_{n} \overset{Q}{\longrightarrow} X$ for some $X \in L^{0}_{c}$. Then there exists a subsequence $(X_{n_{k}})_{k \in \N}$ of $(X_{n})_{n \in \N}$ such that $X_{n_{k}} \longrightarrow X$ $Q$-a.s. For an arbitrary choice $g_{n_{k}} \in X_{n_{k}}$, $k \in \N$, and $g \in X$ set
    \begin{equation*}
        \{g_{n_{k}} \rightarrow g\} := \{\omega \in \Omega \colon \lim_{k\to \infty} g_{n_{k}}(\omega) = g(\omega)\}.
    \end{equation*}
    Note that $Q[\{g_{n_{k}} \to g\}] = 1$ and
    \begin{equation*}
        \forall \omega\in \Omega \quad g_{n_{k}}(\omega) \chi_{\{g_{n_{k}} \rightarrow g\}}(\omega) \to g(\omega) \chi_{\{g_{n_{k}} \rightarrow g\}}(\omega).
    \end{equation*}
    The latter and the fact that every $\tilde Q\in \mathcal Q$ is consistent with the $\mathcal P$-q.s.-order implies 
    \begin{equation*}
        X_{n_{k}} \mathbf{1}_{\{g_{n_{k}} \rightarrow g\}} \overset{\tilde Q}{\longrightarrow} X \mathbf{1}_{\{g_{n_{k}} \rightarrow g\}}
    \end{equation*}
    for all $\tilde Q\in \mathcal Q$. By solidness of $\mathcal{C}$ we have $X_{n_{k}} \mathbf{1}_{\{g_{n_{k}} \rightarrow g\}} \in \mathcal{C}$ for all $k\in \N$, and thus, by  $\mathcal Q$-closedness, $X \mathbf{1}_{\{g_{n_{k}} \rightarrow g\}} \in \mathcal{C}$. Since $Q[\{g_{n_{k}} \to g\}] = 1$ we have $j_{Q}(X) = j_{Q}(X \mathbf{1}_{\{g_{n_{k}} \rightarrow g\}})$. Therefore, $\mathcal{C}$ is locally $Q$-closed.
\end{proof}

One of the more commonly used closedness concepts in robust frameworks is order closedness, see for instance \cite{GM2020} or \cite{LMS2022}.
 
\begin{defi}
    \begin{enumerate}[itemindent=25pt, leftmargin=0pt, nosep]
        \item A set $\mathcal{C} \subseteq L^{0}_{c}$ is order closed if for any net $(X_{\alpha})_{\alpha \in I} \subseteq \mathcal{C}$ and $X \in L^{0}_{c}$ such that $X_{\alpha} \order{c} X$ we have $X \in \mathcal{C}$.
        
        \item A set $\mathcal{C} \subseteq L^{0}_{c}$ is sequentially order closed if for any sequence $(X_{n})_{n \in \N} \subseteq \mathcal{C}$  and $X \in L^{0}_{c}$ such that $X_{n} \order{c} X$ we have $X \in \mathcal{C}$.
    \end{enumerate}
\end{defi}

In the dominated case,  for $Q \in \mathfrak{P}(\Omega)$, we know by the super Dedekind completeness of $L^{0}_{Q}$ (see \cite[Definition~1.43]{AB2003})  that $\mathcal{C}\subset L^{0}_{Q}$ is  order closed if and only if $\mathcal{C}$ is sequentially order closed, and for solid sets this is well-known (see \cite[Lemma 4.1]{LMS2022}) to be equivalent to $Q$-closedness:

\begin{lem} \label{lem:Qclosed:dom}
    Let $Q \in \mathfrak{P}(\Omega)$ and $\mathcal{C}\subseteq L^{0}_{Q}$ be solid. Then the following are equivalent:
    \begin{enumerate}[itemindent=25pt, leftmargin=0pt, nosep]
        \item $\mathcal{C}$ is order closed (with respect to  the $Q$-a.s.\ order).
        
        \item $\mathcal{C}$ is sequentially order closed.
        
        \item $\mathcal{C}$ is $Q$-closed.
    \end{enumerate}
\end{lem}

Possessing some appealing features, in robust frameworks, authors have tended to focus on order convergence as a generalization of $Q$-closedness, see \cite{GM2020} or \cite{LMS2022}. However, it turns out that in the non-dominated case order closedness is generally not equivalent to sequential order closedness, see for instance Examples~\ref{exmp:seqOrdNotOrdCont} and \ref{exmp:pSensi:oClsd}, and that, in fact, it is the latter notion which is closely related to the other natural robustifications of $Q$-closedness we have encountered so far, namely $\mathcal Q$-closedness or local $Q$-closedness for all $Q \in \mathcal{Q}$, see Theorem~\ref{thm:clsd:seqO:allQ:locQ} below. Before we state  Theorem~\ref{thm:clsd:seqO:allQ:locQ} we need two auxiliary results: 

\begin{lem}\label{lem:solid}
    Suppose that $\mathcal{C} \subseteq L^{0}_c$ is solid. Let $Q\in \mathfrak{P}_c(\Omega)$. Then $j_Q(\mathcal C)$ is solid.
\end{lem}
\begin{proof}
    Suppose that $\mathcal{C}\neq \emptyset$ is solid in $L^{0}_c$ and that  $X^Q\in j_Q(\mathcal C)$ and $Y^Q\in L^0_Q$ satisfy $|Y^Q| \leq_Q |X^Q|$. Pick $\tilde X\in \mathcal C$ such that $j_Q(\tilde X)=X^Q$. Further let $f\in \tilde X$ and $g\in Y^Q$ and set $X := [f\chi_{\{|f|\geq |g|\}}]_c$ and $Y := [g\chi_{\{|f|\geq |g|\}}]_c$. Note that $Q[|f|\geq |g|] = 1$ and therefore $j_Q(X) = X^Q$. We have $|Y| \preccurlyeq |X| \preccurlyeq |\tilde X|$, and thus $Y \in \mathcal{C}$. Since $j_Q(Y) = Y^Q$, we conclude that $Y^Q \in j_Q(\mathcal C)$, so $j_Q(\mathcal C)$ is indeed solid with respect to $\leq_Q$. A similar argument applies in the case where $\mathcal{C}$ is solid in $L^{0}_{c+}$.
\end{proof}

\begin{lem} \label{lem:OqsCl:OasCl}
    Suppose that $\emptyset \neq \mathcal{C} \subseteq L^{0}_c$ is solid and sequentially order closed, and let $Q \in \mathfrak{P}_c(\Omega)$. Then $j_Q(\mathcal C)$ is closed with respect to the $Q$-a.s.\ order in $L^0_Q$.
\end{lem}
\begin{proof}
    As $j_Q(\mathcal C)$ is solid according to Lemma~\ref{lem:solid}, in order to show (sequential) order closedness of $j_Q(\mathcal C)$ it suffices to consider non-negative increasing sequences $(X^{Q}_{n})_{n \in \N} \subseteq j_Q(\mathcal{C})$ (that is $0\leq _Q X^{Q}_{n}\leq_Q X^{Q}_{n+1}$ for all $n\in \N$) such that the supremum $X^Q\in L^{0}_{Q}$ of $(X^{Q}_{n})_{n \in \N} $ exists and to show that $X^Q\in j_Q(\mathcal C)$, see \cite[Lemma~1.15]{AB2003}. Pick $(\tilde{X}_{n})_{n \in \N} \subseteq \mathcal{C}$ such that $j_{Q}(\Tilde{X}_{n}) = X^{Q}_{n}$ for all $n\in \N$. Let $g \in X^{Q}$ and $g_{n} \in \Tilde{X}_{n}$ for all $n \in \N$. Consider the event
    \begin{equation*}
        A := \{\sup_{n\in \N} g_n = g\} \cap \{g_1\geq 0\} \cap \bigcap_{n\in \N} \{g_n\leq g_{n+1}\}.
    \end{equation*}
    Note that $Q[A]=1$. Set $X_{n} := [g_{n} \chi_{A}]_c$ for all $n \in \N$ and $X := [g \chi_{A}]_c$. Since $X_n\preccurlyeq \tilde X_n$ we conclude by solidness of $\mathcal C$ that $(X_{n})_{n \in \N} \subseteq \mathcal{C}$. As for all $\omega \in \Omega$ we have $$g_n(\omega)\chi_{A}(\omega)\leq g_{n+1}(\omega)\chi_{A}(\omega)\leq g(\omega)\chi_{A}(\omega), \; n\in \N,$$  and $g(\omega)\chi_{A}(\omega)=\sup_{n\in \N}g_n(\omega)\chi_{A}(\omega)$, we infer that  $X_n\preccurlyeq X_{n+1}\preccurlyeq X$, $n\in \N$, and  $X=\sup_{n\in \N} X_n$ in $(L^0_c,\preccurlyeq)$. Consequently, $X_n\order{c} X$.  
    Hence, by sequential order closedness of $\mathcal C$ we obtain $X\in \mathcal C$. Now  $j_Q(X)=X^Q$ implies $X^Q\in j_Q(\mathcal{C})$.
\end{proof}

\begin{thm} \label{thm:clsd:seqO:allQ:locQ}
    Suppose that $\mathcal{C} \subseteq L^{0}_{c}$ is solid and $\mathcal{P}$-sensitive. Let $\mathcal{Q} \subseteq \mathfrak{P}_c(\Omega)$ be a reduction set for $\mathcal C$. Then the following are equivalent:
    \begin{enumerate}[itemindent=25pt, leftmargin=0pt, nosep]
        \item $\mathcal{C}$ is sequentially order closed.
        
        \item $\mathcal{C}$ is $\mathcal{Q}$-closed.
        
        \item $\mathcal{C}$ is locally $Q$-closed for each $Q \in \mathcal{Q}$.
        
        \item $j_Q(\mathcal{C})$ is $Q$-closed in $L^0_Q$ for each $Q \in \mathcal{Q}$.
        
        \item $j_Q(\mathcal{C})$ is order closed with respect to the $Q$-a.s.\ order on $L^0_Q$ for each $Q \in \mathcal{Q}$.
        
        \item $j_Q(\mathcal{C})$ is sequentially order closed with respect to the $Q$-a.s.\ order on $L^0_Q$ for each $Q \in \mathcal{Q}$.
    \end{enumerate}
\end{thm}

For the proof of Theorem~\ref{thm:clsd:seqO:allQ:locQ} we need another auxiliary lemma:

\begin{lem} \label{lem:help1}
    Let $(X_n)_{n\in \N}\subseteq L^0_c$ and $Q\in \mathfrak{P}_c(\Omega)$.
    \begin{enumerate}[itemindent=25pt, leftmargin=0pt, nosep]
        \item Suppose that the  infimum (supremum) $X=\inf_{n\in \N} X_n$ ($X=\sup_{n\in \N} X_n$)  of $(X_n)_{n\in \N}$ in the $\mathcal P$-q.s.\ order  exists. Then, $j_Q(X)=\inf_{n\in \N}j_Q(X_n)$ ($j_Q(X)=\sup_{n\in \N}j_Q(X_n)$) in $L^0_Q$, i.e., $j_Q(X)$ is the infimum (supremum) of $(j_Q(X_n))_{n\in \N}$ in the $Q$-a.s.\ order. 
        
        \item Let $Y\in L^0_c$ and suppose that $X_n\order{c}Y$ in $L^0_c$. Then, $j_Q(X_n)\order{Q}j_Q(Y)$ in $L^0_Q$.
    \end{enumerate}
\end{lem}
\begin{proof}
    \begin{enumerate}[itemindent=25pt, leftmargin=0pt, nosep]
        \item We only prove the case of the infimum. From $Q\ll \mathcal P$ it immediately follows that $j_Q(X)$ is a lower  bound for $(j_Q(X_n))_{n\in \N}$. Consider another lower bound $Z^Q\in L^0_Q$ of $(j_Q(X_n))_{n\in \N}$.  We have to show that $j_Q(X) \geq_Q Z^Q$. For any choice $f_n\in X_n$ and $g\in Z^Q$ we have that $Q[\{f_n\geq g\}] = 1$ and thus also the event
        \begin{equation*}
            A := \bigcap_{n\in \N} \{f_n \geq g\}
        \end{equation*}
        satisfies $Q[A]=1$. Let $Z:=[g]_c\mathbf{1}_{A} + X\mathbf{1}_{A^c}\in L^0_c$. Then $Z\preccurlyeq X_n$, $n\in \N$, and hence $Z\preccurlyeq X$ which implies $j_Q(X)\geq_Q j_Q(Z)=Z^Q$.
        
        \item By definition of order convergence, there exists a decreasing sequence $(Y_{n})_{n \in \N} \subseteq L^{0}_{c+}$ such that $\inf_{n \in \N} Y_{n} = 0$ in $L^0_c$ and for all $n \in \N$
        \begin{equation*}
            \lvert X_{n} - X \rvert \preccurlyeq Y_{n}.
        \end{equation*}
        Define $X^{Q} := j_{Q}(X)$ and $X^{Q}_{n} := j_{Q}(X_{n})$, $Y^{Q}_{n} := j_{Q}(Y_{n})$, $n\in \N$. As $Q \ll \mathcal{P}$, we have for all $n \in \N$
        \begin{equation*}
            \lvert X^{Q}_{n} - X^{Q} \rvert \leq_Q Y^{Q}_{n} \quad \mbox{and} \quad 
             0\leq_QY^{Q}_{n + 1}\leq _QY^{Q}_{n}. 
        \end{equation*}
        According to 1. $\inf_{n\in \N}Y^Q_n=0$ in $L^0_Q$. Hence, $X^Q_n\order{Q} X^Q$.
    \end{enumerate}
\end{proof}

\begin{proof}[Proof of Theorem~\ref{thm:clsd:seqO:allQ:locQ}]
    $2. \ \Leftrightarrow \ 3.\ \Leftrightarrow \ 4.$: see Lemma~\ref{lem:loc:clsd} and Proposition~\ref{prop:clsd:allQ:loc}.
    
    $4. \ \Leftrightarrow \ 5.\ \Leftrightarrow \ 6.$: follow from Lemma~\ref{lem:Qclosed:dom}.
    
    $1. \ \Rightarrow \ 6.$: Lemma~\ref{lem:OqsCl:OasCl}.
    
    $6. \ \Rightarrow \ 1.$: Assume $\mathcal C\neq \emptyset$ and let  $(X_{n})_{n \in \N} \subseteq \mathcal{C}$ such that $X_{n} \order{c} X \in L^{0}_{c}$. According to Lemma~\ref{lem:help1}, $j_Q(X_n)\order{Q} j_Q(X)$. As $j_Q(\mathcal{C})$ is closed in the $Q$-a.s.\ order for any $Q\in \mathcal{Q}$ we obtain $j_{Q}(X) \in j_Q(\mathcal{C})$ for all $Q \in \mathcal{Q}$. Since $\mathcal{Q}$ is a reduction set for $\mathcal{C}$ we infer $X \in \mathcal{C}$.     
\end{proof}

Interestingly, also in the robust case there are situations in which we may add order closedness to the list in Theorem~\ref{thm:clsd:seqO:allQ:locQ}. This is closely related to the existence of supports of probability measures as introduced in Section~\ref{sec:support}.

\begin{lem} \label{lem:qsOasO}
    Let $Q\in \mathfrak{P}_c(\Omega)$.
    \begin{enumerate}[itemindent=25pt, leftmargin=0pt, nosep]
        \item Suppose that $Q$ is supported. Let $\mathcal C\subseteq L^0_c$ and suppose that the infimum (supremum) $X:=\inf \mathcal C$ ($X:=\sup \mathcal C$) exists in $L^0_c$. Then $j_Q(X)= \inf j_Q(\mathcal{C})$ ($j_Q(X)= \sup j_Q(\mathcal{C})$) in $L^0_Q$. In particular, for any net $(X_\alpha)_{\alpha\in I}\subseteq L^0_c$ and $X\in L^0_c$ we have that $X_\alpha\order{c} X$ implies $j_Q(X_\alpha)\order{Q} j_Q(X)$. 
        
        \item Conversely, suppose that for any net $(X_\alpha)_{\alpha\in I}\subseteq L^0_c$ and $X\in L^0_c$ we have that $X_\alpha\order{c} X$ implies $j_Q(X_\alpha)\order{Q} j_Q(X)$, then $Q$ is supported.
    \end{enumerate}
\end{lem}
\begin{proof}
    \begin{enumerate}[itemindent=25pt, leftmargin=0pt, nosep]
        \item We only prove the case of the infimum. From $Q \ll \mathcal P$ it immediately follows that $j_Q(X)$ is a lower  bound for $j_Q(\mathcal{C})$. Hence, we only have to show that any lower bound $Y^Q\in L^0_Q$ of $j_Q(\mathcal C)$ in $L^0_Q$ satisfies $j_Q(X)\geq_Q Y^Q$. Denote by $S(Q)$ a version of the $Q$-support. Similar to the proof of Lemma~\ref{lem:help1} we pick $f\in Y^Q$ and define $Y:=[f]_c\mathbf{1}_{S(Q)} + X1_{S(Q)^c}$. We have that $Y \preccurlyeq Z$  for all $Z\in \mathcal C$. Indeed, let $Z\in \mathcal C$ and $g\in Z$ (and thus also $g\in j_Q(Z)$). Since $0 = Q[f > g] = Q[S(Q) \cap \{f > g\}]$ we infer that $c[S(Q) \cap \{f > g\}] = 0$ (recall Definition~\ref{defi:supp}). Therefore $[f]_c \mathbf{1}_{S(Q)} \preccurlyeq  Z\mathbf{1}_{S(Q)}$. $X$ being a lower bound of $\mathcal C$ now yields $Y \preccurlyeq Z$. As $Z\in \mathcal C$ was arbitrary, and as $X$ is the largest lower bound of $\mathcal C$, we conclude that $Y\preccurlyeq X$. This in turn implies that $j_Q(X)\geq_Q j_Q(Y)=Y^Q$ where we have used that $Q[S(Q)] = 1$ for the latter equality. The remaining part of the assertion now follows along similar lines as presented in the proof of Lemma~\ref{lem:help1}.

        \item Note that by the dominated convergence theorem, for any measure $P\in  \mathfrak{P}(\Omega)$, the linear functional
        \begin{equation*}
            l_P \colon L^\infty_P \ni X \mapsto E_P[X]
        \end{equation*}
        is always $\sigma$-order continuous and thus also order continuous, because $L^\infty_P $ is super Dedekind complete. Under the assumption stated in 2.\ we thus have that
        \begin{equation*}
            L^\infty_c \ni X \mapsto E_Q[X],
        \end{equation*}
        which we may view as the composition $l_Q\circ j_Q$, is order continuous. Since the order continuous dual of $L^\infty_c$ may be identified with $sca_c$, see \cite[Proposition B.3]{LMS2022}, we find that $Q$ must be supported.
    \end{enumerate} 
\end{proof}

Combining Theorem~\ref{thm:clsd:seqO:allQ:locQ} with Lemma~\ref{lem:qsOasO} we obtain:

\begin{thm} \label{thm:clsd:ord:seqO:allQ:locQ}
    Suppose that $\mathcal{C} \subseteq L^{0}_{c}$ is solid and $\mathcal{P}$-sensitive and let $\mathcal{Q} \subseteq \mathfrak{P}_c(\Omega) \cap sca_{c}$ be a reduction set for $\mathcal C$. Then the following are equivalent:
    \begin{enumerate}[itemindent=25pt, leftmargin=0pt, nosep]
        \item $\mathcal{C}$ is order closed.
        
        \item $\mathcal{C}$ is sequentially order closed.
        
        \item $\mathcal{C}$ is $\mathcal{Q}$-closed.
        
        \item $\mathcal{C}$ is locally $Q$-closed for each $Q \in \mathcal{Q}$.
        
        \item $j_Q(\mathcal{C})$ is $Q$-closed in $L^0_Q$ for each $Q \in \mathcal{Q}$.
        \item $j_Q(\mathcal{C})$ is order closed with respect to the $Q$-a.s.\ order on $L^0_Q$ for each $Q \in \mathcal{Q}$.
        \item $j_Q(\mathcal{C})$ is sequentially order closed with respect to the $Q$-a.s.\ order on $L^0_Q$ for each $Q \in \mathcal{Q}$.
    \end{enumerate}
\end{thm}
\begin{proof}
    In the view of Theorem~\ref{thm:clsd:seqO:allQ:locQ} and as obviously $1. \ \Rightarrow \ 2.$, it suffices to prove that  $6. \ \Rightarrow \ 1.$ But if $\mathcal C\neq \emptyset$ and $(X_\alpha)_{\alpha\in I}\subseteq \mathcal C$ and $X\in L^0_c$ satisfy $X_\alpha\order{c} X$, then $j_Q(X_\alpha)\order{Q} j_Q(X)$ for all $Q\in \mathcal{Q}$ according to Lemma~\ref{lem:qsOasO}. Thus, 6.\ implies that $j_Q(X)\in j_Q(\mathcal{C})$ for all $Q\in \mathcal Q$, and $\mathcal Q$ being a reduction set for $\mathcal C$ now yields $X\in \mathcal C$.
\end{proof}

Note that in Theorem~\ref{thm:clsd:ord:seqO:allQ:locQ} it is important that we consider a reduction set $\mathcal Q$ for $\mathcal C$ which is strictly smaller than $\mathfrak{P}_c(\Omega)$ if $ca_c\neq sca_c$. In fact, $ca_c\neq sca_c$ is often the case according to \cite[Section 3.3]{LMS2022}, see also Example~\ref{exmp:classS}. In the sequel we will encounter more situations in which the existence of a suitable reduction set with further properties than $\mathfrak{P}_c(\Omega)$ is crucial.

The following example shows that the equivalence 1.\ $\Leftrightarrow$ 2.\ in Theorem~\ref{thm:clsd:ord:seqO:allQ:locQ} generally does not hold if the reduction set is not supported:

\begin{exmp} \label{exmp:seqOrdNotOrdCont}
    Recall that $ca_{c}$ and $sca_{c}$ can be identified with the $\sigma$-order and the order continuous dual of $L^{\infty}_{c}$, respectively, see, for instance, \cite{LMS2022}. Let us assume that $sca_{c} \neq ca_{c}$ (see Example~\ref{exmp:classS}) and let $\mu \in ca_{c+} \setminus sca_{c+}$. Consider 
    \begin{equation*}
        \mathcal{C}_r := \bigg\{X \in L^{\infty}_{c+} \colon \int X d\mu \leq r\bigg\}
    \end{equation*}
    where $r>0$. $\mathcal{C}_r$ is obviously convex and solid. Moreover,  $\mathcal{C}_r$ is $\mathcal{P}$-sensitive with reduction set $\mathcal Q=\{Q\}$ where $Q := \mu[\Omega]^{-1} \mu \in \mathfrak{P}_{c}(\Omega)$ is not supported. Since $L^\infty_c\ni X\mapsto \int X d\mu $ is not order continuous, there exists a decreasing net $(X_\alpha)_{\alpha\in I}\subseteq L^\infty_{c+}$ with $\inf_{\alpha\in I} X_\alpha =0 $ such that  $\inf_{\alpha\in I} \int X_\alpha d\mu =: b>0 $. Let $\beta \in I$. Then the net $Y_\alpha:=X_\beta-X_\alpha$, $\alpha\geq \beta$, is increasing and satisfies $0\preccurlyeq Y_\alpha$ and  $Y_\alpha \order{c} X_\beta$. We have $(Y_\alpha)_{\alpha\geq \beta}\subseteq \mathcal{C}_r $  for $r=\int X_\beta d\mu-b$, but $X_\beta\not\in \mathcal{C}_r$. Hence, $\mathcal{C}_r$ is sequentially order closed but not order closed.
\end{exmp}

\section{$\mathcal P$-Sensitivity Reloaded} \label{sec:pSensi:rel}

In this section we study necessary and sufficient conditions for $\mathcal{P}$-sensitivity of $\mathcal C\subseteq L^0_c$.  We start with some rather evident structural properties. 

\subsection{$\mathcal P$-Sensitivity by Local Defining Conditions}

\begin{prop}\label{prop:sensitivity:local}
    Let $\emptyset\neq \mathcal{Q} \subseteq \mathfrak{P}_c(\Omega)$ and fix a quantifier $\dagger\in \{\exists, \forall\}$. Suppose that 
    \begin{equation*}
        \mathcal{C} = \bigcap_{Q \in \mathcal{Q}} \{X \in L^{0}_{c} \colon \dagger H \in \mathcal{H} \ Q[A^{H}_{Q}(X)] = 1\},
    \end{equation*}
    where $\mathcal{H}$ is a non-empty set and for all $H\in \mathcal{H}$ and $Q\in \mathcal Q$ the function $A^{H}_{Q} \colon L^{0}_{c} \to \mathcal{F}$ satisfies
    \begin{equation*}
        Q[A^{H}_{Q}(X) \triangle A^{H}_{Q}(Y)] = 0,
    \end{equation*}
    whenever $Q[X = Y] = 1$. Then, $\mathcal{C}$ is $\mathcal{P}$-sensitive with reduction set $\mathcal{Q}$.
\end{prop}
\begin{proof}
    Assume $\mathcal C\neq \emptyset$ and let $X \in L^{0}_{c}$ such that $j_{Q}(X) \in j_Q(\mathcal{C})$ for all $Q \in \mathcal{Q}$. Fix $Q \in \mathcal{Q}$. Then there exists $Y \in \mathcal{C}$ such that $j_{Q}(X) = j_{Q}(Y)$, that is $Q[X = Y] = 1$. Hence, depending on the quantifier,
    \begin{equation*}
      \dagger H \in \mathcal{H} \quad Q[A^{H}_{Q}(X)] = Q[A^{H}_{Q}(Y)] = 1.
    \end{equation*}
    As $Q \in \mathcal{Q}$ was arbitrary, $X \in \mathcal{C}$.
\end{proof}

\begin{exmp}\label{ex:sensitive:sets}
    Let $\mathcal{Q} \subseteq \mathfrak{P}_c(\Omega)$ be non-empty. To conclude $\mathcal{P}$-sensitivity of the following examples of sets $\mathcal{C}$ we apply Proposition~\ref{prop:sensitivity:local} in conjunction with the facts that $L^0_{c+}$ is $\mathcal{P}$-sensitive (see Example~\ref{exmp:pSensi}) and that the intersection of $\mathcal{P}$-sensitive sets remains $\mathcal{P}$-sensitive (see Lemma~\ref{lem:intersect:Psensi}).
    \begin{enumerate}[itemindent=25pt, leftmargin=0pt, nosep]
        \item (local boundedness condition)  Set $\mathcal{H} := \N$ and $A^{n}_{Q}(X) := \{\omega \in \Omega \colon f(\omega) \leq n\}$ for some $f\in X$ and $n \in \N$. Then
        \begin{equation*}
            \mathcal{C} := \{X \in L^{0}_{c+} \colon  \forall Q \in \mathcal Q \, \exists n \in \N \ Q[X \leq n] = 1\}
        \end{equation*}
        is $\mathcal{P}$-sensitive and even convex and solid. However, $\mathcal{C}$ is not sequentially order closed, as we can easily see that $\mathcal C$ is not $\mathcal{Q}$-closed.
        
        \item (uniform local boundedness condition) Let $Y_{Q} \in L^{0}_{c+}$ for each $Q \in \mathcal{Q}$. Set $\mathcal{H} := \{0\}$ and $A^{0}_{Q}(X) := \{\omega \in \Omega \colon f(\omega) \leq g(\omega)\}$ for some $f \in X$ and $g \in Y_{Q}$. Then
        \begin{align*}
            \mathcal{C} :=  \{X \in L^{0}_{c+} \colon \forall Q \in \mathcal{Q} \ Q[X \leq Y_{Q}] = 1\}
        \end{align*}
        is $\mathcal{P}$-sensitive, convex, and solid. Clearly, $\mathcal{C}$ is also $\mathcal{Q}$-closed and hence sequentially order closed.

        \item (uniform martingale condition) Let $\mathcal H:=\{(Y,\mathcal{G})\}$ for some sub-$\sigma$-algebra $\mathcal{G}$ of $\mathcal{F}$ and some $Y\in L^0_{c+}$ which admits a $\mathcal{G}$-measurable representative $g\in Y$.  The set
        \begin{equation*}
            \mathcal{C} := \{X \in L^{0}_{c+} \colon \forall Q \in \mathcal Q \; \forall f \in E_Q[X\mid \mathcal G] \ f = g \ Q\text{-a.s.}\}
        \end{equation*}
        is $\mathcal P$-sensitive. Here, $E_Q[X\mid \mathcal G]\in L^0_c(\Omega, \mathcal G, Q)$ denotes the equivalence class of conditional expectations under $Q$ of (any representative of) $X$ given $\mathcal G$. Indeed, let
        \begin{equation*}
            A^{(Y,\mathcal{G})}_{Q}(X) := \{\omega \in \Omega \colon f(\omega) = g(\omega)\}
        \end{equation*}
        for some arbitrary choice $f\in E_Q[X\mid \mathcal G]$. Then
        \begin{equation*}
            \mathcal{C} := \{X \in L^{0}_{c+} \colon \forall Q \in \mathcal Q \ Q[A^{(Y,\mathcal{G})}_{Q}(X)] = 1\}
        \end{equation*}

        \item (uniform supermartingale condition) Again let $\mathcal H:=\{(Y,\mathcal{G})\}$ for some sub-$\sigma$-algebra $\mathcal{G}$ of $\mathcal{F}$ and some $Y\in L^0_{c+}$ which admits a $\mathcal{G}$-measurable representative $g\in Y$.  The set
        \begin{equation*}
            \mathcal{C} := \{X \in L^{0}_{c+} \colon \forall Q \in \mathcal Q \; \forall f \in E_Q[X \mid \mathcal G] \ f \leq g \; Q\text{-a.s.}\}
        \end{equation*}
        is $\mathcal P$-sensitive ($A^{(Y,\mathcal{G})}_{Q}(X) := \{\omega \in \Omega \colon f(\omega) \leq g(\omega)\}$ for some arbitrary choice $f \in E_Q[X \mid \mathcal G]$). Moreover, $\mathcal{C}$ is solid, convex, $\mathcal{Q}$-closed. Hence, by Theorems~\ref{thm:clsd:seqO:allQ:locQ} and \ref{thm:clsd:ord:seqO:allQ:locQ}, $\mathcal{C}$ is sequentially order closed and even order closed if $\mathcal{Q}\subseteq sca_c$.

        \item Let $Y\in L^0_{c+}$. Then the set
        \begin{equation*}
            \mathcal C := \{X \in L^0_{c+} \colon X \preccurlyeq Y\} = \{X \in L^0_{c+} \colon \forall P \in \mathcal P \  P[X \leq Y] = 1\}
        \end{equation*}
        is convex, solid, and sequentially order closed. $\mathcal C$ is also $\mathcal P$-sensitive according Proposition~\ref{prop:sensitivity:local}.  Indeed, set $\mathcal H:=\{Y\}$ and $A^{Y}_P(X):=\{\omega \in \Omega \colon f(\omega)\leq g(\omega)\}$, $P \in \mathcal P =\mathcal Q$, $X \in L^0_c$, where $f \in X$ and $g \in Y$.
    \end{enumerate}
\end{exmp}

The following lemma is easily verified. We will apply it in our discussion of robust acceptability criteria in Section~\ref{sec:appli:acc}.

\begin{lem}\label{lem:Psensi:acc}
    Let $\emptyset\neq \mathcal{Q} \subseteq \mathfrak{P}_c(\Omega)$, and for each $Q\in \mathcal{Q}$ fix a set $\mathcal{C}_Q\in L^0_{Q}$. Then $$\mathcal{C}=\bigcap_{Q\in \mathcal{Q}}j_Q^{-1}(\mathcal{C}_Q)$$ is $\mathcal{P}$-sensitive with reduction set $\mathcal{Q}$.
\end{lem}
\begin{proof}
    Let $X \in L^0_c$ satisfy $j_Q(X)\in j_Q(\mathcal{C})$ for all $Q \in \mathcal{Q}$. Then, for each $Q \in \mathcal{Q}$, there is $Y^Q \in \mathcal{C}$ such that $Q[Y^Q = X] = 1$. As $Y^Q \in \mathcal{C}$ we have that $j_Q(Y^Q) \in \mathcal{C}_Q$. $Q[Y^Q = X] = 1$ implies that $j_Q(Y^Q) = j_Q(X)$, and therefore $j_Q(X) \in \mathcal{C}_Q$. As $Q \in \mathcal{Q}$ was arbitrary, $\mathcal{P}$-sensitivity of $\mathcal{C}$ with reduction set $\mathcal{Q}$ follows. 
\end{proof}

The next result is a slight modification of the observation already made in Corollary~\ref{cor:l0:bipolar}

\begin{lem}\label{lem:Psensi:dual}
    Let $\mathcal{M}\subseteq ca_{c}\setminus \{0\}$ be non-empty. Further let $$\mathcal{C}:=\bigg\{X\in L^0_c \colon \forall \mu\in \mathcal{M} \ \text{$X$ is $\mu$-integrable and $\int Xd\mu \leq a_\mu$}\bigg\},$$ where $a_\mu\in \R$, $\mu\in  \mathcal{M}$. Then $\mathcal{C}$ is $\mathcal{P}$-sensitive with reduction set $\mathcal{Q}:=\{|\mu|/|\mu|(\Omega)\mid \mu \in \mathcal{M}\}$. 
\end{lem}
\begin{proof}
    Let $X\in L^0_c$ such that $j_Q(X) \in j_{Q}(\mathcal{C})$ for all $Q \in \mathcal Q$, and let $\mu \in \mathcal{M}$. For $Q := \frac{\lvert \mu \rvert}{\lvert \mu \rvert(\Omega)} \in \mathcal{Q}$ pick $Y \in \mathcal{C}$ such that $j_{Q}(X) = j_{Q}(Y)$. As $Q[X = Y] = 1$, $X$ is $\mu$-integrable, and
    \begin{equation*}
        \int X d\mu  = \int Y d\mu \leq a_\mu.
    \end{equation*}
    Since $\mu \in \mathcal{M}$ was arbitrary, we infer that $X \in \mathcal{C}$.
\end{proof}

\subsection{$\mathcal P$-Sensitivity and Aggregation}
In the following we relate $\mathcal{P}$-sensitivity to the concept of aggregation (cf. \cite[Section 1.5]{HP1974, T1991}).
\begin{defi}
    Let $\mathcal{Q} \subseteq \mathfrak{P}_{c}(\Omega)$.
    \begin{enumerate}[itemindent=25pt, leftmargin=0pt, nosep]
        \item A family $(X^{Q})_{Q \in \mathcal{Q}} \subseteq L^{0}_{c}$ is $\mathcal{Q}$-coherent if there is $X^{\mathcal{Q}} \in L^{0}_{c}$ such that
        \begin{equation*}
            \forall Q \in \mathcal{Q} \quad Q[X^{Q} = X^{\mathcal{Q}}] = 1.
        \end{equation*}
        The equivalence class $X^{\mathcal{Q}}$ is called a $\mathcal{Q}$-aggregator of $(X^{Q})_{Q \in \mathcal{Q}}$.
        
        \item  A set $\mathcal C\subseteq L^0_c$ is called $\mathcal Q$-stable if for any $\mathcal{Q}$-coherent family $(X^{Q})_{Q \in \mathcal{Q}} \subseteq \mathcal C$ the set $\mathcal C$ contains all $\mathcal Q$-aggregators of $(X^{Q})_{Q \in \mathcal{Q}}$. 
    \end{enumerate}
\end{defi}

\begin{exmp}\label{exmp:stability}
Recall Example~\ref{exmp:classS}. For $Q=\delta_\omega$ let $X^Q={\bf 1}_{\{\omega\}}$, $\omega\in \Omega$. The family $({\bf 1}_{\{\omega\}})_{\omega\in \Omega}$ is $\mathcal{Q}$-coherent with $\mathcal{Q}$-aggregator $1$. The set $\mathcal{C} = \{{\bf 1}_{\{\omega\}} \colon \omega \in \Omega\}$ is not $\mathcal{Q}$-stable since $1 \not \in \mathcal{C}$. However, the set $\mathcal{D} = \{X\in L^0_c \colon X \preccurlyeq 1\}$ is $\mathcal{Q}$-stable. Indeed, consider any $\mathcal{Q}$-coherent family $(X^{Q})_{Q \in \mathcal{Q}} \subseteq \mathcal{D}$ and let $X^{\mathcal{Q}} \in L^{0}_{c}$ be a $\mathcal{Q}$-aggregator of $(X^{Q})_{Q \in \mathcal{Q}}$ (for instance $({\bf 1}_{\{\omega\}})_{\omega\in \Omega}$ and $1$). Then, for each $Q \in \mathcal{Q}$, $Q[X^{Q} = X^{\mathcal{Q}}] = 1$. It follows that $Q[X^{\mathcal{Q}} \leq 1] = Q[X^{Q} \leq 1] = 1$ for every $Q \in \mathcal{Q}$, and thus $X^{\mathcal{Q}}\preccurlyeq 1$ (recall $\mathcal{Q}\approx \mathcal{P}$). Therefore, $X^{\mathcal{Q}} \in \mathcal{D}$, and $\mathcal{D}$ is $\mathcal{Q}$-stable.
\end{exmp}

\begin{exmp}\label{exmp:stability:2} 
Similarly to the $\mathcal{Q}$-stability of $\mathcal{D}$ in Example~\ref{exmp:stability}, one verifies that the sets $\mathcal{C}_-^a$ and $\mathcal{C}_+^a$ introduced in Example~\ref{exmp:pSensi} are  $\mathcal{P}$-stable. 
\end{exmp}

 The $\mathcal{P}$-stable sets $\mathcal{C}_-^a$ and $\mathcal{C}_+^a$ from Example~\ref{exmp:stability:2} are also $\mathcal{P}$-sensitive according to Example~\ref{exmp:pSensi}. This is no surprise as the following proposition shows.

\begin{prop} \label{prop:pSensi:aggre}
    Let $\mathcal{Q} \subseteq \mathfrak{P}_{c}(\Omega)$. Then, a non-empty set $\mathcal{C} \subseteq L^{0}_{c}$ is $\mathcal{P}$-sensitive with reduction set $\mathcal{Q}$ if and only if $\mathcal C$ is $\mathcal{Q}$-stable. 
\end{prop}
\begin{proof}
    Let $\mathcal{C}$ be $\mathcal{P}$-sensitive with reduction set $\mathcal{Q}$. Suppose that $(X^{Q})_{Q \in \mathcal{Q}} \subseteq \mathcal{C}$ is $\mathcal{Q}$-coherent and let $X^{\mathcal{Q}} \in L^{0}_{c}$ be a $\mathcal{Q}$-aggregator. Then $j_{Q}(X^{\mathcal{Q}}) = j_{Q}(X^{Q}) \in \mathcal{C}_{Q}$ for all $Q \in \mathcal{Q}$. Hence, as $\mathcal Q$ is a reduction set for $\mathcal{C}$, $X^{\mathcal{Q}} \in \mathcal{C}$. Thus $\mathcal{C}$ is $\mathcal{Q}$-stable. 

    Now suppose that $\mathcal{C}$ is $\mathcal{Q}$-stable. Let $X \in \bigcap_{Q \in \mathcal{Q}} j_{Q}^{-1} \circ j_{Q}(\mathcal{C})$. Then there exist $(X^{Q})_{Q \in \mathcal{Q}} \subseteq \mathcal{C}$ such that $j_Q(X^Q) = j_Q(X)$, that is $Q[X = X^{Q}] = 1$, for all $Q \in \mathcal{Q}$. Thus, $X$ is a $\mathcal{Q}$-aggregator for $(X^{Q})_{Q \in \mathcal{Q}} \subseteq \mathcal{C}$ and therefore $X \in \mathcal{C}$. Hence, $\mathcal{C}$ is $\mathcal{P}$-sensitive with reduction set $\mathcal{Q}$.
\end{proof}

\begin{exmp} \label{exmp:pSensi:aggre}
    Suppose that the (multivariate) process $S$ in continuous or discrete time describes the discounted price evolution of some financial assets. Let $\mathcal{H}$ be a set of investment strategies and denote the portfolio wealth at terminal time $T>0$ of some $H\in \mathcal{H}$ as $(H \cdot S)_{T}$, which is a random variable. The latter will typically coincide with a stochastic integral at time $T$, and $(H \cdot S)_{0}=0$. The set of superhedgeable claims at cost less than $1$ is given by 
    \begin{equation} \label{eq:superhedge:C}
        \mathcal{C} := \{X \in L^{0}_{c+} \colon \exists H \in \mathcal{H} \ X \preccurlyeq 1 + (H \cdot S)_{T}\}.
    \end{equation}
    It is well-known (see e.g.\ \cite{BKN2021,KS1999, KS2003}) that a bipolar representation of $\mathcal{C}$ is closely related to the set of martingale measures, i.e., probability measures under which the discounted price process $S$ is a martingale, see also Section~\ref{sec:superhedging}. Hence, we are interested in criteria which ensure that $\mathcal{C}$ is $\mathcal P$-sensitive. Indeed, according to Proposition~\ref{prop:pSensi:aggre}, $\mathcal{C}$ is $\mathcal P$-sensitive if and only if $\mathcal{C}$ is $\mathcal{Q}$-stable for some $\mathcal Q\subseteq \mathfrak{P}_c(\Omega)$. This however requires some aggregation property of the portfolio wealths $(H \cdot S)_{T}$. For instance, suppose that $\mathcal{P}$ is of class (S) and that $L^0_c$ is Dedekind complete. The latter assumptions are for instance satisfied in the volatility uncertainty framework discussed in \cite{C2012} and \cite{STZ2011} where Dedekind completeness of $L^0_c$ is achieved by a suitable enlargement of the filtration (\cite[Section 5]{STZ2011} and \cite[Example 5.1]{LMS2022}). Let $\mathcal Q$ be a disjoint supported alternative to $\mathcal P$, see Lemma~\ref{lem:disjoint:alt}. Then any family $(X^Q)_{Q\in \mathcal Q}\subseteq \mathcal C$ is $\mathcal{Q}$-coherent, see Lemma~\ref{lem:Qcoherent} below. Let $X$ be a $\mathcal Q$-aggregator of $(X^Q)_{Q\in \mathcal Q}$ and let $H^Q\in \mathcal H$ be such that $X^Q\preccurlyeq 1 + (H^Q \cdot S)_{T}$. Consider any $\mathcal Q$-aggregator $Y$ of the terminal wealths $((H^Q \cdot S)_{T})_{Q\in \mathcal Q}$, which exists by Lemma~\ref{lem:Qcoherent}. Then
    \begin{equation*}
        X \preccurlyeq 1 + Y.
    \end{equation*}
    A sufficient condition for $\mathcal P$-sensitivity is thus that any such $\mathcal Q$-aggregator of terminal wealths $Y$ can be replicated, that is, there is $H\in \mathcal H$ such that $Y=(H\cdot S)_T$. Indeed, in the case of volatility uncertainty the latter problem is related to finding an aggregator $H$ of the processes $H^Q$, $Q\in \mathcal{Q}$, in the sense of \cite[Definition 3.1]{STZ2011}. This problem is easily solved if the disjoint supports of all probability measures $Q\in \mathcal{Q}$ were $\mathcal{F}_0$-measurable. Then, one could simply paste the processes $H^Q$, $Q\in \mathcal{Q}$, along the supports. More generally, \cite[Theorem 5.1 and Theorem 6.5]{STZ2011} give sufficient conditions for the existence of such an aggregator.
\end{exmp}

\begin{rem}
Note that the set of superhedgeable claims $\mathcal{C}$ given in \eqref{eq:superhedge:C} cannot be handled with Proposition~\ref{prop:sensitivity:local}. The reason is that while   
    \begin{equation*}
        \mathcal{C} = \{X \in L^{0}_{c+} \colon \exists H \in \mathcal{H} \,  \forall P \in \mathcal{P} \ P[X \leq 1 + (H \cdot S)_{T}] = 1\},
    \end{equation*}
Proposition~\ref{prop:sensitivity:local} only allows to conclude $\mathcal{P}$-sensitivity of, for instance,  
\begin{equation*}
        \mathcal{D} = \{X \in L^{0}_{c+} \colon \forall P \in \mathcal{P}\, \exists H \in \mathcal{H} \ P[X \leq 1 + (H \cdot S)_{T}] = 1\}.
    \end{equation*}
    The crucial difference is the order of the quantifiers $\exists H\in \mathcal{H}\, \forall P\in \mathcal{P}$ versus $\forall P\in \mathcal{P}\, \exists H\in \mathcal{H}$ in the conditions defining $\mathcal{C}$ and $\mathcal{D}$, respectively. Indeed, a statement like $\forall P\in \mathcal{P}\, \exists H\in \mathcal{H} \ldots$ enables a local argument under each $P\in \mathcal{P}$, that is the principle of $\mathcal{P}$-sensitivity, whereas $\exists H\in \mathcal{H}\, \forall P\in \mathcal{P}\ldots $ imposes a uniform condition over all probability measures $P\in \mathcal{P}$.
\end{rem}

\begin{lem} \label{lem:Qcoherent}
    Suppose that $\mathcal{P}$ is of class (S) and $L^{0}_{c}$ is Dedekind complete. Let $\mathcal Q $ denote a disjoint supported alternative to $\mathcal P$ (Lemma~\ref{lem:disjoint:alt}). Then any choice $(X^{Q})_{Q \in \mathcal{Q}} \subseteq L^{0}_{c+}$ is $\mathcal{Q}$-coherent. Moreover, any $\mathcal Q$-aggregator $X$ of $(X^{Q})_{Q \in \mathcal{Q}}$ satisfies $X{\bf 1}_{S(Q)}=X^Q{\bf 1}_{S(Q)}$ for all $Q\in \mathcal Q$.
\end{lem}
\begin{proof}
    The last assertion follows from $X{\bf 1}_{S(Q)}=X^Q{\bf 1}_{S(Q)}$ if and only if $Q[X = X^Q] = 1$. For the first assertion let $(X^{Q})_{Q \in \mathcal{Q}} \subseteq L^{0}_{c}$. For $n\in \N$ let $X^n\in L^0_{c+}$ denote the least upper bound of the bounded family $(X^{Q}\wedge n){\bf 1}_{S(Q)}$, $Q\in \mathcal Q$. It then follows that $Q[X^n = X^Q \wedge n] = 1$ for all $Q\in \mathcal{Q}$ (see Lemma~\ref{lem:qsOasO}) and thus \[X^n{\bf 1}_{S(Q)}=(X^Q\wedge n){\bf 1}_{S(Q)}\preccurlyeq X^{Q}{\bf 1}_{S(Q)}.\] Therefore, $X^n\preccurlyeq X^{n+1}$ for all $n\in \N$, and the $\mathcal P$-quasi sure limit $X:=\lim_{n\to \infty} X^n\in L^0_c$ exists, where $(X_n)\subseteq L^0_c$ is said to converge to $X\in L^0_c$ $\mathcal P$-quasi surely if $P[X_n \to X] = 1$ for all $P\in \mathcal P$, and satisfies
    \[X \mathbf{1}_{S(Q)}=X^Q \mathbf{1}_{S(Q)}.\]
    Hence, $X$ is a $\mathcal Q$-aggregator of $(X^{Q})_{Q \in \mathcal{Q}}$. 
\end{proof}

\subsection{$\mathcal P$-sensitivity as a Consequence of Weak Closedness}

Recall the following classical bipolar theorem (see, e.g., \cite[Theorem~5.103]{AB2006}) for locally convex topologies.

\begin{thm}
    Let $\langle \mathcal{X}, \mathcal{Y} \rangle$ be a dual pair, see \cite[Definition~5.90]{AB2006}, and let $\emptyset \neq \mathcal{C} \subseteq \mathcal{X}$. Define $\mathcal{C}^\circ := \{Y \in \mathcal{Y} \colon \forall X \in \mathcal C \ \langle X, Y \rangle \leq 1\}$ and $\mathcal{C}^{\circ\circ}:=\{X \in \mathcal{X} \colon \forall Y \in \mathcal C^\circ \ \langle X, Y \rangle \leq 1\}$. $\mathcal{C} = \mathcal{C}^{\circ\circ}$ if and only if $\mathcal C$ is convex, $\sigma(\mathcal{X}, \mathcal{Y})$-closed, and $0\in \mathcal{C}$.
\end{thm}

The following result shows that if $\mathcal{C}\subseteq \mathcal{X}$ is convex and $\sigma(\mathcal{X}, \mathcal{Y})$-closed with respect to some dual pair $\langle \mathcal{X},\mathcal{Y}\rangle$, where $\mathcal{X}\subseteq L^0_c$ and $\mathcal{Y} \subseteq \mathrm{ca}_{c}$, then $\mathcal{C}$ is essentially $\mathcal{P}$-sensitive.

\begin{thm}\label{thm:weak:pSensi} 
    Let  $\mathcal{X}\subseteq L^0_c$ and $\mathcal{Y} \subseteq \mathrm{ca}_{c}$ be subspaces such that $\langle \mathcal{X},\mathcal{Y}\rangle$ is a dual pair, and $\mathcal{Y}$ satisfies $\mu \in \mathcal{Y} \, \Rightarrow |\mu|\in \mathcal{Y}$.
    Suppose that $\mathcal{C} \subseteq \mathcal{X}$ is non-empty, convex, and $\sigma(\mathcal{X}, \mathcal{Y})$-closed. Then there is $\mathcal{Q}\subseteq \mathfrak{P}_{c}(\Omega)\cap  \mathcal Y$ such that
    \begin{equation} \label{eq:Psensitive:X}
        \mathcal{C} = \bigcap_{Q\in \mathcal{Q}}j^{-1}_Q\circ j_Q(\mathcal{C}) \cap \mathcal{X}.
    \end{equation} 
\end{thm}

The property \eqref{eq:Psensitive:X} may be viewed as $\mathcal{C}$ being \emph{$\mathcal{P}$-sensitive in $\mathcal{X}$ with reduction set $\mathcal{Q}$}.

\begin{proof}
    The convex indicator function $f \colon \mathcal{X} \rightarrow [0, \infty]$ defined as
    \begin{equation*}
        f(X) := \delta(X \mid \mathcal{C}) =
        \begin{cases}
            0, &X \in \mathcal{C}, \\
            \infty, &X \not\in \mathcal{C},
        \end{cases}
    \end{equation*}
    is convex and $\sigma(\mathcal{X}, \mathcal{Y})$-lower semi-continuous. By the Fenchel-Moreau theorem, 
    \begin{equation*}
        f(X) = f^{**}(X) = \sup_{\mu \in \mathcal{Y}} \int X d\mu - f^{*}(\mu)
    \end{equation*} where $f^\ast :\mathcal Y\to (-\infty, \infty]$ is given by \[f^\ast (\mu)=\sup_{X\in \mathcal X}  \int X d\mu - f(X).\]
    We may thus represent $\mathcal C$ as
    \begin{align*}
        \mathcal{C} &= \{X \in \mathcal{X} \colon f(X) = 0\} = \bigcap_{\mu \in \mathrm{dom} f^{*} \setminus \{0\}} \{X \in \mathcal{X} \colon \int X d\mu - f^{*}(\mu) \leq 0\},
    \end{align*}
    where $\mathrm{dom} f^{*}:=\{\mu \in \mathcal Y \colon f^\ast (\mu)<\infty\}$, and the last equality uses the fact that for $\mu = 0$
    \begin{equation*}
        f^{*}(\mu)  = - \inf_{Y \in \mathcal{X}} f(Y) = 0 = \int X d\mu
    \end{equation*}
    for all $X \in \mathcal{X}$. Let $\mathcal{Q} := \{\frac{\lvert \mu \rvert}{\lvert \mu \rvert(\Omega)} \colon \mu \in \mathrm{dom} f^{*} \setminus \{0\}\}$.  Then $\mathcal{Q}\subseteq \mathcal{Y}\cap \mathfrak{P}_c(\Omega)$.     
    Now \eqref{eq:Psensitive:X} follows similar to the proof of Lemma~\ref{lem:Psensi:dual}.
\end{proof}

\begin{cor}
    In the situation of Theorem~\ref{thm:weak:pSensi} suppose that $\mathcal{X}$ is $\mathcal{P}$-sensitive with reduction set $\mathcal{Y}\cap \mathfrak{P}_c(\Omega)$, then $\mathcal{C}$ is $\mathcal{P}$-sensitive with reduction set $\mathcal{Y}\cap \mathfrak{P}_c(\Omega)$. 
\end{cor}
\begin{proof}
    This follows from \eqref{eq:Psensitive:X} and Lemma~\ref{lem:intersect:Psensi}.
\end{proof}

The next simple example shows that even in a dominated framework the $\mathcal P$-sensitive sets in $L^0_c$ do not all coincide with weakly closed sets in some locally convex subspace $\mathcal X$ of $L^0_c$.

\begin{exmp}
    Let $\mathcal{P} = \{P\}$ for a non-atomic probability measure $P \in \mathfrak{P}(\Omega)$. In this case, it is well-known that there is no subspace $\mathcal Y\subseteq ca_P \simeq L^1_{P}$ such that  $\langle L^{0}_{P}, \mathcal{Y} \rangle$ is a dual pair. Indeed, for any $\mu\in ca_P \setminus \{0\}$ there is $X\in L^0_{P+} $ such that $ \int X d \mu$ is not well-defined or infinite. However, $\mathcal{C} := L^{0}_{P+}$ is convex, solid, and trivially $P$-sensitive with reduction set $\{P\}$. Also, $\mathcal{C}$ admits a bipolar representation with polar set $\mathcal{C}^{\circ} = \{\mu \in ca_{P+} \colon \forall X \in \mathcal C \ \int X d\mu \leq 1 \} = \{0\}$ and $\mathcal{C}^{\circ\circ} = \{X \in L^{0}_{c+} \colon 0 \leq 1\} = L^{0}_{c+} = \mathcal{C}$, see Section~\ref{sec:bip:thm}. 
\end{exmp}

\subsection{$\mathcal P$-Sensitivity as a Consequence of Class (S) and Order Closedness}

As mentioned previously, a widely used closedness requirement in robust frameworks is order closedness, see \cite{GM2020} and \cite{LMS2022}. Supposing that $\mathcal P$ is of class (S), we will in the following show that order closedness  combined with convexity and solidness already implies $\mathcal{P}$-sensitivity.

\begin{lem}
    Suppose that $\mathcal{P}$ is of class (S) and let $\mathcal Y\subseteq sca_{c}$ be any linear space separating the points of $L^\infty_c$, i.e., for any $X,Y\in L^\infty_c$ such that $X\neq Y$ there is $\mu \in \mathcal{Y}$ such that $\int Xd\mu \neq \int Yd\mu$. Moreover, let $\mathcal{C}\subseteq L^0_{c}$ be convex, solid, and order closed. Then $\mathcal{C} \cap L^{\infty}_{c}$ is $\sigma(L^{\infty}_{c}, \mathcal{Y})$-closed.
\end{lem}
\begin{proof}
    Let $\tau := \lvert\sigma\rvert(L^{\infty}_{c}, \mathcal{Y})$ be the absolute weak topology on $L^{\infty}_{c}$ generated by $\mathcal{Y}$ (see \cite[Definition~2.32]{AB2003}). The topology $\tau$ is locally convex-solid and Hausdorff with the Lebesgue property (see \cite[Definition~3.1]{AB2003}) since $sca_c$ may be identified with the order continuous dual of $L^\infty_c$ (see, e.g., \cite{LMS2022}). Suppose $\mathcal{C}\neq \emptyset$. Consider the set $\mathcal{D} := \mathcal{C} \cap L^{\infty}_{c}$. $\mathcal{D}$ is non-empty (because for each $X \in \mathcal{C}$ and $k \in \N$, $-k\vee X \wedge k  \in \mathcal{D}$ by solidness), convex, solid, and order closed. Using \cite[Lemma~4.2 and Lemma~4.20]{AB2003}, we infer that $\mathcal{D}$ is $\lvert\sigma\rvert(L^{\infty}_{c},\mathcal{Y})$-closed. As $\lvert\sigma\rvert(L^{\infty}_{c},\mathcal{Y})$ and $\sigma(L^{\infty}_{c}, \mathcal{Y})$ share the same closed convex sets (see \cite[Theorem~8.49 and Corollary~5.83]{AB2006}), $\mathcal{D}$ is $\sigma(L^{\infty}_{c}, \mathcal{Y})$-closed.
\end{proof}

\begin{cor} \label{cor:oCLsd:sensi}
    Suppose that $\mathcal{P}$ is of class (S) and let $\mathcal Y\subseteq sca_c$ be a subspace separating the points of $L^\infty_c$ such that $\mu \in \mathcal{Y} \, \Rightarrow |\mu|\in \mathcal{Y}$. Further assume that $\mathcal{C}\subseteq L^0_{c}$ is convex, solid, and order closed. Then $\mathcal{C}$ is $\mathcal{P}$-sensitive with reduction set $ \mathfrak{P}_c(\Omega) \cap \mathcal{Y}$.
\end{cor}
Note that, when $\mathcal{P}$ is of class (S), $sca_c$ always separates the points of $L^\infty_c$, see \cite[Proposition B.5]{LMS2022}.

\begin{proof} $\langle L^\infty_c , \mathcal{Y}\rangle$ is a dual pair.
    The previous lemma shows that $\mathcal{C} \cap L^{\infty}_{c}$ is $\sigma(L^{\infty}_c, \mathcal{Y})$-closed. According to Theorem \ref{thm:weak:pSensi}, there is $\mathcal{Q}\subseteq \mathcal{Y}\cap \mathfrak{P}_c(\Omega)$ such that $\mathcal{C}\cap L^\infty_c$ satisfies \eqref{eq:Psensitive:X} where $\mathcal{X}= L^\infty_c$. Let 
    \begin{equation*}
        X \in \bigcap_{Q \in \mathcal{Q}} j_{Q}^{-1} \circ j_{Q}(\mathcal{C}).
    \end{equation*}

For all $Q\in \mathcal{Q}$ there is $Y\in \mathcal{C}$ such that $j_Q(Y)=j_Q(X)$. 
As $\mathcal{C}$ is solid, for $n\in \N$, we have $-n\vee Y\wedge n\in \mathcal{C}\cap L^\infty_c$, which in particular implies $j_Q(-n\vee X\wedge n)=j_Q(-n\vee Y\wedge n) \in j_{Q}(\mathcal{C})$. As $Q\in \mathcal{Q}$ was arbitrary, and by \eqref{eq:Psensitive:X}, we have that $-n\vee X\wedge n\in \mathcal{C}$ for all $n\in \N$. Finally, order closedness of $\mathcal{C}$ implies $X\in \mathcal{C}$.   
\end{proof}

The next example, which can originally be found in \cite[Example 3.7]{MMS2018}, gives us an example of a convex, solid, and sequentially order closed set which is not $\mathcal{P}$-sensitive (and thus not order closed). Moreover, $\mathcal P$ will be of class (S) and $L^0_c$ will be Dedekind complete. However, the example is based on assuming the continuum hypothesis, i.e., there is no set $\mathfrak{X}$ whose cardinality satisfies $\lvert \N \rvert  = \aleph_{0} < \lvert \mathfrak{X} \rvert < 2^{\aleph_{0}} = \lvert \R \rvert$.

\begin{exmp} \label{exmp:pSensi:oClsd}
    Let $(\Omega, \mathcal{F}) = ([0, 1], \mathbb{P}([0, 1]))$, where $\mathbb{P}([0, 1])$ denotes the power set of $[0, 1]$. Further, let $\mathcal{P} := \{\delta_{\omega} \colon \omega \in [0, 1]\}$ be the set of all Dirac measures. Every probability measure in $\mathcal P$ is supported (see Example~\ref{exmp:classS}), and $L^0_c$ is easily seen to be Dedekind complete. \\
    Consider the set
    \begin{equation*}
        \mathcal{D} := \{\mathbf{1}_{A} \colon \emptyset \neq A \subseteq [0, 1] \text{ is countable}\}
    \end{equation*}
    and let $\mathcal{C}$ be the solid hull of $\mathcal{D}$ in $L^0_{c+}$. $\mathcal{C}$ can be written as
    \begin{equation*}
        \mathcal{C} = \{X \in L^{0}_{c+} \colon \exists Y \in \mathcal{D} \  0 \leq X \leq Y\},
    \end{equation*}
    which also shows convexity of $\mathcal{C}$. Note that every $X\in \mathcal C$ is countably supported. Now let $(X_{n})_{n \in \N} \subseteq \mathcal{C}$ such that $X_{n} \overset{o}{\underset{c}{\longrightarrow}} X \in L^{0}_{c+}$. For each $X_{n} \in \mathcal{C}$ there exists a countable set $A_{n} \subseteq [0, 1]$ such that $0 \leq X_{n} \leq \mathbf{1}_{A_{n}}$. Set $A := \bigcup_{n \in \N} A_{n}$. $A$ is still countable and  $0 \leq X_{n} \leq \mathbf{1}_{A}$ for all $n \in \N$. Hence, $0 \leq X \leq \mathbf{1}_{A}$ and therefore $X \in \mathcal{C}$. Thus, $\mathcal{C}$ is sequentially order closed. \\
    Next we show that $\mathcal C$ is not order closed. To this end,  set $I := \{A \subseteq [0, 1] \text{ finite}\}$. For $\alpha, \beta \in I$ we let $\alpha \leq \beta$ if $\alpha \subseteq \beta$ and set $X_\alpha=\mathbf{1}_\alpha$. Then $(X_{\alpha})_{\alpha \in I}\subseteq \mathcal{C}$ converges in order to $1$, but $ 1  \not\in \mathcal{C}$. Hence, $\mathcal{C}$ is not order closed. Indeed, so far we have reproduced a well-known text book example of a set which is sequentially order closed, but not order closed, and we have not yet assumed the continuum hypothesis. \\
    In order to give an answer to the question whether $\mathcal{C}$ is $\mathcal{P}$-sensitive, we need more structure. From now on assume the continuum hypothesis. Banach and Kuratowski have shown that for any set $\Lambda$ with the same cardinality as $\R$ there is no measure $\mu$ on $(\Lambda, \mathbb{P}(\Lambda))$, where $\mathbb{P}(\Lambda)$ denotes the power set of $\Lambda$,  such that $\mu(\Lambda) = 1$ and $\mu(\{\omega\}) = 0$ for all $\omega \in \Lambda$, see for instance \cite[Theorem C.1]{D2002}. It follows that any probability measure $\mu$ on $(\Omega, \mathcal{F})$ must be a countable sum of weighted Dirac-measures, i.e., $\mu = \sum_{i = 1}^{\infty} a_{i} \delta_{\omega_{i}}$, where $\sum_{i = 1}^{\infty} a_{i} = 1$, $a_{i} \geq 0$, and $\omega_{i} \in \Omega$ for all $i \in \N$. Thus every probability measure has a countable support, and in particular $ca = ca_c = sca = sca_c$. \\
    Let $Q \in \mathfrak{P}_{c}(\Omega) = \mathfrak{P}(\Omega)$. $Q$ has a countable support $S(Q)$, and therefore $\mathbf{1}_{S(Q)} \in \mathcal{C}$. Then $j_{Q}(1) = j_{Q}(\mathbf{1}_{S(Q)}) \in j_Q(\mathcal{C})$. As $Q \in \mathfrak{P}_{c}(\Omega)$ was arbitrary, we have
    \begin{equation*}
        1 \in \bigcap_{Q \in \mathfrak{P}_{c}(\Omega)} j_{Q}^{-1} \circ j_{Q}(\mathcal{C}).
    \end{equation*}
    However, as before, $1 \not\in \mathcal{C}$. Hence, $\mathcal{C}$ is not $\mathcal{P}$-sensitive.
\end{exmp}

Example \ref{exmp:pSensi:oClsd} also implies that there is no proof of the statement that convexity, solidness, and sequential order closedness imply $\mathcal{P}$-sensitivity:

\begin{cor}
    Let $\mathcal{C} \subseteq L^{0}_{c+}$ be convex, solid, and sequentially order closed. Without further assumptions, there exists no proof that the assumed properties of $\mathcal{C}$ imply $\mathcal{P}$-sensitivity.
\end{cor}
\begin{proof}
    This follows from Example~\ref{exmp:pSensi:oClsd} and the fact that the continuum hypothesis is consistent with the standard mathematical axioms (ZFC).
\end{proof}

\section{Bipolar Theorems in $L^0_{c+}$} \label{sec:bip:thm}

Based on Theorem~\ref{thm:bipolar:ext} we will now lift Theorem~\ref{thm:KS} to $L^0_{c+}$.

\begin{thm} \label{thm:l0:bipolar:ext:ks}
    Suppose that $\mathcal{C} \subseteq L^0_{c+}$ and $\mathcal{Q}\subseteq \mathfrak{P}_c(\Omega)$ are non-empty . Let 
    \begin{equation*}
        \mathcal{C}^{\circ}_\mathcal{Q} := \{(Q, Z) \in \mathcal{Q} \times L^{\infty}_{c+} \colon \forall X \in \mathcal{C} \ E_{Q}[Z X] \leq 1\}
    \end{equation*}
    and 
    \begin{equation*}
        \mathcal{C}^{\circ \circ}_\mathcal{Q} := \{X \in L^{0}_{c+} \colon \forall (Q, Z) \in \mathcal{C}^{\circ}_\mathcal{Q} \ E_{Q}[Z X] \leq 1\}.
    \end{equation*}
    Then $\mathcal{C}^{\circ \circ}_\mathcal{Q}$ is convex, solid, sequentially order closed, and $\mathcal P$-sensitive with reduction set $\mathcal{Q}$, and $\mathcal{C}\subseteq \mathcal{C}^{\circ \circ}_\mathcal{Q}$. Moreover, $\mathcal{C}^{\circ \circ}_\mathcal{Q}$ is the smallest such set in the sense that any set $\mathcal{D}\subseteq L^0_{c+}$ containing $\mathcal{C}$, which is also convex, solid, sequentially order closed, and $\mathcal P$-sensitive with reduction set $\mathcal{Q}$, satisfies $\mathcal{C}^{\circ \circ}_\mathcal{Q}\subseteq \mathcal{D}$.  In particular, $\mathcal{C}=\mathcal{C}^{\circ \circ}_\mathcal{Q}$ if and only if $\mathcal{C}$ is convex, solid, sequentially order closed, and $\mathcal P$-sensitive with reduction set $\mathcal{Q}$.
\end{thm}
\begin{proof}
    Clearly, $\mathcal{C}\subseteq \mathcal{C}^{\circ \circ}_\mathcal{Q}$, and $\mathcal P$-sensitivity with reduction set $\mathcal{Q}$, convexity, solidness, and sequentially order closedness of $\mathcal{C}^{\circ \circ}_\mathcal{Q}$ follow from Corollary~\ref{cor:l0:bipolar} and Theorem~\ref{thm:clsd:seqO:allQ:locQ}.
    
    Now suppose that $\mathcal{C}$ is $\mathcal P$-sensitive with reduction set $\mathcal{Q}$, convex, solid, and sequentially order closed. Consider any $Q \in \mathcal{Q}$. $j_Q(\mathcal{C})$ is clearly convex in $L^{0}_{Q+}$ and also solid by Lemma~\ref{lem:solid}.  Moreover, by  Theorem~\ref{thm:clsd:seqO:allQ:locQ}, $j_Q(\mathcal{C})$ is $Q$-closed. Hence, according to Theorem~\ref{thm:KS}, the requirement of Theorem~\ref{thm:bipolar:ext} is satisfied, and we have
    \begin{equation*}
        \mathcal{C} = \{X \in L^0_{c+} \colon \forall Q \in \mathcal{Q} \, \forall Z_Q \in j_Q(\mathcal{C})^\circ \ E_Q[Z_Q j_Q(X)] \leq 1\},
    \end{equation*}
    where $j_Q(\mathcal{C})^\circ\subseteq L^\infty_{Q+}$ is the polar of $j_Q(\mathcal C)$ given in Theorem~\ref{thm:KS}. Fix $Q\in \mathcal{Q}$ and $Z_Q\in j_Q(\mathcal{C})^\circ$. Then for any $ Z\in j_Q^{-1}(Z_Q)\cap L^\infty_{c+}$ and all $X\in L^0_{c+}$  we have $E_Q[Z_Qj_Q(X)]=E_Q[ZX]$. In particular, $E_Q[ZX]= E_Q[Z_Qj_Q(X)]\leq 1$ for all $X\in \mathcal{C}$ because $j_Q(X)\in j_Q(\mathcal{C})$ and $Z_Q\in j_Q(\mathcal{C})^\circ$. Hence, $(Q,Z)\in \mathcal{C}^\circ_\mathcal{Q}$. Therefore, if $X\in L^0_{c+}$ satisfies
    \begin{equation*}
        \forall (Q, Z) \in \mathcal{C}^{\circ}_\mathcal{Q} \quad E_{Q}[Z X] \leq 1,
    \end{equation*}
    then $X$ satisfies
    \begin{equation*}
        \forall Z_Q \in j_Q(\mathcal{C})^\circ \quad E_Q[Z_Qj_Q(X)]\leq 1,
    \end{equation*}
    and thus  $\mathcal{C}_{\mathcal{Q}}^{\circ\circ}\subseteq \mathcal{C}$. Since always $\mathcal{C}\subseteq \mathcal{C}_{\mathcal{Q}}^{\circ\circ}$ we have $\mathcal{C}= \mathcal{C}_{\mathcal{Q}}^{\circ\circ}$.

    Minimality of $\mathcal{C}^{\circ\circ}_\mathcal{Q}$ follows by standard arguments. Indeed, suppose that $\mathcal{D}$ is $\mathcal P$-sensitive with reduction set $\mathcal{Q}$, convex, solid, and sequentially order closed, and that $\mathcal{C}\subseteq \mathcal{D}$.  The latter implies $\mathcal{D}^\circ_\mathcal{Q} \subseteq \mathcal{C}^\circ_\mathcal{Q}$, and therefore $\mathcal{C}^{\circ\circ}_\mathcal{Q}\subseteq \mathcal{D}^{\circ\circ}_\mathcal{Q}$. We have shown above that $\mathcal{D}^{\circ\circ}_\mathcal{Q}=\mathcal{D}$, that is $\mathcal{C}^{\circ\circ}_\mathcal{Q}\subseteq \mathcal{D}$.     
\end{proof}

For the generic case $\mathcal{Q}=\mathfrak{P}_c(\Omega)$ we write
\begin{equation} \label{eq:c:circ}
    \mathcal{C}^{\circ} := \mathcal{C}^{\circ}_{\mathfrak{P}_c(\Omega)} \quad \text{and} \quad \mathcal{C}^{\circ\circ} := \mathcal{C}^{\circ\circ}_{\mathfrak{P}_c(\Omega)}.
\end{equation}
Note that $\mathcal{C}^{\circ}_{\mathcal{Q}} \subseteq \mathcal{C}^{\circ}$ for any non-empty $\mathcal{Q} \subseteq \mathfrak{P}_c(\Omega)$ and thus $\mathcal{C}^{\circ\circ} \subseteq \mathcal{C}^{\circ\circ}_{\mathcal{Q}}$. 

\begin{cor}\label{cor:them:KS:ext} Suppose that  $\mathcal{C}\subseteq L^0_{c+}$ is non-empty. $\mathcal{C}^{\circ\circ}$ is the smallest convex, solid, sequentially order closed, $\mathcal P$-sensitive subset of $L^0_{c+}$ containing $\mathcal C$. $\mathcal{C}=\mathcal{C}^{\circ \circ}$ if and only if $\mathcal{C}$ is convex, solid, sequentially order closed, and $\mathcal P$-sensitive. If, moreover, $\mathcal{Q}\subseteq \mathfrak{P}_c(\Omega)$ is a reduction set for $\mathcal{C}$, then $\mathcal{C}=\mathcal{C}^{\circ\circ} = \mathcal{C}^{\circ\circ}_\mathcal{Q}$.
\end{cor}
\begin{proof}
This follows from Theorem~\ref{thm:l0:bipolar:ext:ks} since $\mathfrak{P}_{c}(\Omega)$ is a (in fact the maximal) reduction set for any $\mathcal P$-sensitive set.
\end{proof}

\begin{cor} Let $\mathcal{C}\subseteq L^0_{c+}$ be non-empty.
    If the non-empty set $\mathcal Q\subseteq sca_c$ has disjoint supports, then
    \begin{equation*}
        \mathcal{C}^{\circ\circ}_\mathcal{Q} = \mathcal{C}^{\ast \ast}_{\mathcal{Q}} := \{X \in L^{0}_{c+} \colon \forall Z \in \mathcal{C}^{\ast}_{\mathcal Q} \ \sup_{Q \in \mathcal{Q}} E_{Q}[Z X] \leq 1\}
    \end{equation*}
    where
    \begin{equation*}
        \mathcal{C}^{\ast}_{\mathcal Q} := \{Z \in L^\infty_{c+} \colon \forall X \in \mathcal{C} \ \sup_{Q \in \mathcal{Q}} E_Q[Z X] \leq 1\}.
    \end{equation*}
\end{cor}

\begin{proof}
    As $$\mathcal{C}^{\ast\ast}_{\mathcal Q} = \{X \in L^{0}_{c+} \colon \forall Z \in \mathcal{C}^{\ast}_{\mathcal Q} \, \forall Q \in \mathcal{Q} \ E_{Q}[Z X] \leq 1\},$$ Corollary~\ref{cor:l0:bipolar} and Theorem~\ref{thm:clsd:seqO:allQ:locQ} show that $\mathcal{C}^{\ast\ast}_{\mathcal Q}$ is  $\mathcal P$-sensitive with reduction set $\mathcal{Q}$, convex, solid, and sequentially order closed. Moreover, $\mathcal{C}^{\ast\ast}_{\mathcal Q}$ contains $\mathcal C$. Therefore, by Theorem~\ref{thm:l0:bipolar:ext:ks} $\mathcal{C}^{\circ\circ}_{\mathcal Q}\subseteq \mathcal{C}^{\ast\ast}_{\mathcal Q}$. It remains to show that $\mathcal{C}^{\ast\ast}_{\mathcal Q} \subseteq \mathcal{C}^{\circ \circ}_{\mathcal{Q}}$. To this end, let $X\in \mathcal{C}^{\ast\ast}_{\mathcal Q}$. For any  $(Q,Z)\in\mathcal{C}^{\circ}_{\mathcal Q}$ we have $Z\mathbf{1}_{S(Q)}\in \mathcal{C}^{\ast}_{\mathcal Q}$. Indeed, by disjointness of the supports, we obtain
    \begin{equation*}
        \sup_{\tilde Q\in \mathcal{Q}} E_{\tilde Q}[Z\mathbf{1}_{S(Q)}Y] = E_{Q}[Z\mathbf{1}_{S(Q)}Y] = E_{Q}[ZY] \leq 1
    \end{equation*}
    for all $Y\in \mathcal{C}$. Hence, $Z\mathbf{1}_{S(Q)}\in \mathcal{C}^{\ast}_{\mathcal{Q}}$ and thus
    \begin{equation*}
        E_Q[Z X] = \sup_{\tilde Q\in \mathcal{Q}} E_{\tilde Q}[Z\mathbf{1}_{S(Q)}X] \leq 1.
    \end{equation*}
    As $(Q,Z)\in \mathcal{C}^{\circ}_{\mathcal Q}$ was arbitrary, this implies $X\in  \mathcal{C}^{\circ \circ}_{\mathcal{Q}}$.
\end{proof}

Analogously to the proof of Theorem~\ref{thm:l0:bipolar:ext:ks} we can obtain a lifting of Theorem~\ref{thm:BS}, or we simply conclude it from Theorem~\ref{thm:l0:bipolar:ext:ks}: 

\begin{thm} \label{thm:l0:bipolar:ext}
    Suppose that $\mathcal{C} \subseteq L^0_{c+}$ and $\mathcal{Q}\subseteq \mathfrak{P}_{c}(\Omega) $ are non-empty. Let
    \begin{equation*}
        \mathcal{C}^{\diamond}_\mathcal{Q} := \{(Q, Z) \in \mathcal{Q} \times L^{0}_{c+} \colon \forall X \in \mathcal{C} \ E_{Q}[Z X] \leq 1\}
    \end{equation*}
    and 
    \begin{equation*}
         \mathcal{C}^{\diamond \diamond}_\mathcal{Q} := \{X \in L^{0}_{c+} \colon \forall (Q, Z) \in \mathcal{C}^{\diamond}_\mathcal{Q} \ E_{Q}[Z X] \leq 1\}.
    \end{equation*}
    Then $\mathcal{C}^{\diamond \diamond}_\mathcal{Q}=\mathcal{C}^{\circ \circ}_\mathcal{Q}$ where $\mathcal{C}^{\circ\circ}_\mathcal{Q}$ is given in Theorem~\ref{thm:l0:bipolar:ext:ks}.
\end{thm}
\begin{proof}
    We have $\mathcal{C}\subseteq \mathcal{C}^{\diamond \diamond}_\mathcal{Q}\subseteq \mathcal{C}^{\circ\circ}_\mathcal{Q}$ since $\mathcal{C}^{\circ}_\mathcal{Q}\subseteq \mathcal{C}^{\diamond}_\mathcal{Q}$. According to Corollary~\ref{cor:l0:bipolar} and Theorem~\ref{thm:clsd:seqO:allQ:locQ}, $\mathcal{C}^{\diamond \diamond}_\mathcal{Q}$ is $\mathcal{P}$-sensitive with reduction set $\mathcal{Q}$, convex, solid, and sequentially order closed. Therefore, $\mathcal{C}^{\diamond \diamond}_\mathcal{Q}=\mathcal{C}^{\circ \circ}_\mathcal{Q}$  by Theorem~\ref{thm:l0:bipolar:ext:ks}.
\end{proof}

Note that the polar in Theorem~\ref{thm:l0:bipolar:ext} involves unbounded random variables. The advantage of the bipolar representation based on bounded random variables in Theorem~\ref{thm:l0:bipolar:ext:ks}, compared to Theorem~\ref{thm:l0:bipolar:ext}, is that it implies a representation over finite measures: 

\begin{cor} \label{cor:l0:bipolar:meas}
    Suppose that $\mathcal{C}\subseteq L^0_{c+}$ is non-empty. Let
    \begin{equation*}
        \mathcal{C}^{\circ\circ}_{ca} := \bigg\{X \in L^{0}_{c+} \colon \forall \mu \in \mathcal{C}^\circ_{ca} \ \int X d\mu \leq 1\bigg\}
    \end{equation*}
    where
    \begin{equation*}
        \mathcal{C}^\circ_{ca} := \bigg\{\mu \in ca_{c+} \colon \forall X \in \mathcal{C} \ \int X d\mu \leq 1\bigg\}.
    \end{equation*}
    Then $\mathcal{C}^{\circ\circ}_{ca}=\mathcal{C}^{\circ\circ}$ where $\mathcal{C}^{\circ\circ}$ is defined in \eqref{eq:c:circ}. In particular, $\mathcal{C}=\mathcal{C}^{\circ\circ}_{ca}$ if and only if $\mathcal{C}$ is convex, solid, sequentially order closed, and $\mathcal P$-sensitive.

    Furthermore, if $\mathcal{C}$ is convex, solid, sequentially order closed, and $\mathcal P$-sensitive with reduction set $\mathcal{Q} \subseteq sca_{c}$, then
    \begin{equation*}
        \mathcal{C}=\mathcal{C}^{\circ\circ} = \mathcal{C}^{\circ\circ}_{sca} := \bigg\{X \in L^{0}_{c+} \colon \forall \mu \in \mathcal{C}^\circ_{sca} \ \int X d\mu \leq 1\bigg\},
    \end{equation*}
    where
    \begin{equation*}
       \mathcal{C}^\circ_{sca} := \bigg\{\mu \in sca_{c+} \colon \forall X \in \mathcal{C} \ \int X d\mu \leq 1\bigg\}.
    \end{equation*}
    Both $\mathcal{C}^\circ_{ca}$ and $\mathcal{C}^\circ_{sca}$ are convex, solid, and $\sigma(ca_c,L^\infty_c)$-closed or $\sigma(sca_c,L^\infty_c)$-closed, respectively. Here solid means that $\mu \in \mathcal{C}^\circ_{ca}$ (resp.\ $\mu \in \mathcal{C}^\circ_{sca}$) and $\nu \in ca_{c+}$ (resp.\ $\nu \in sca_{c+}$) such that $\nu[A]\leq \mu[A]$ for all $A \in \mathcal{F}$ imply $\nu \in \mathcal{C}^\circ_{ca}$ (resp.\ $\nu\in \mathcal{C}^\circ_{sca}$)
\end{cor}
\begin{proof}
    Note that any $(Q, Z) \in \mathfrak{P}_c(\Omega)\times L^\infty_{c+}$  can be identified with a measure $\mu \in ca_c$ given by $\mu[A] = E_Q[Z \mathbf{1}_A]$, $A \in \mathcal{F}$. Hence, if $(Q,Z)\in \mathcal{C}^{\circ}$ it follows that  $\mu \in \mathcal{C}^\circ_{ca}$, and therefore
    $\mathcal{C}^{\circ\circ}_{ca} \subseteq \mathcal{C}^{\circ\circ}$. 
    $\mathcal{C}^{\circ\circ}_{ca}$ contains $\mathcal{C}$, and $\mathcal{C}^{\circ\circ}_{ca}$ is convex and solid, and also sequentially order closed by the monotone convergence theorem. $\mathcal P$-sensitivity of $\mathcal{C}^{\circ\circ}_{ca}$ is shown in Proposition~\ref{prop:pSensi}. Hence, $\mathcal{C}^{\circ\circ}_{ca}=\mathcal{C}^{\circ\circ}$ follows from Corollary~\ref{cor:them:KS:ext}. The latter then also implies that $\mathcal{C}=\mathcal{C}^{\circ\circ}_{ca}$ if and only if $\mathcal{C}$ is convex, solid, sequentially order closed, and $\mathcal P$-sensitive.
    
    If $\mathcal C$ is $\mathcal P$-sensitive with reduction set $\mathcal{Q}\subseteq sca_c$, then, similar to the previous considerations, we obtain $\mathcal{C} \subseteq \mathcal{C}^{\circ\circ}_{sca} \subseteq \mathcal{C}^{\circ\circ}_\mathcal{Q}$. If $\mathcal{C}$ is also convex, solid, and sequentially order closed, then $\mathcal{C} =\mathcal{C}^{\circ\circ}= \mathcal{C}^{\circ\circ}_\mathcal{Q}$ by Corollary~\ref{cor:them:KS:ext}, so we must have $\mathcal{C} =\mathcal{C}^{\circ\circ}=\mathcal{C}^{\circ\circ}_{sca}$.
     
    Convexity of $\mathcal{C}^\circ_{ca}$ and $\mathcal{C}^\circ_{sca}$ is easily verified. Regarding solidness, note that if $\nu,\mu\in ca_{c+}$ are such that $\nu[A] \leq \mu[A]$ for all $A \in \mathcal{F}$, then $\int X d\nu \leq \int X d\mu$ for all $X\in L^0_{c+}$. We proceed to prove $\sigma(ca_c,L^\infty_c)$-closedness of $\mathcal{C}^\circ_{ca}$: Consider a net $(\mu_\alpha)_{\alpha\in I}\subseteq \mathcal{C}^\circ_{ca}$ such that $\mu_\alpha\to \mu$ with respect to $\sigma(ca_c,L^\infty_c)$. For all $X\in \mathcal C$ and all $n\in \N$ we have $\int (X\wedge n) d\mu_\alpha \leq \int X d\mu_\alpha \leq 1$ by monotonicity of the integral. Moreover, $$\int (X\wedge n)d\mu = \lim_\alpha \int (X\wedge n) d\mu_\alpha\leq 1 $$ since $(X\wedge n)\in L^\infty_c$. As $\mu\in ca_{c+}$, the monotone convergence theorem now implies $\int X d\mu \leq 1$. Hence, $\mu \in \mathcal{C}^\circ_{ca}$. The same argument shows  $\sigma(sca_c,L^\infty_c)$-closedness of $\mathcal{C}^\circ_{sca}$.
\end{proof}

Finally, we give the following standard result on $\mathcal{C}_{ca}^\circ$ which will be needed in Section~\ref{sec:mass:transport}.

\begin{lem} \label{lem:polar:aux}
    Let $\mathcal{M}\subseteq ca_{c+}$ be non-empty and define
    \begin{equation*}
        \mathcal{C}:=\bigg\{X\in L^0_{c+} \colon \forall \mu \in\mathcal{M} \ \int X d\mu \leq 1\bigg\}.
    \end{equation*}
    Then $\mathcal{C}_{ca}^\circ$ is the smallest solid, convex, and $\sigma(ca_c,L^\infty_c)$-closed subset of $ca_{c+}$ containing $\mathcal{M}$.
    
    The same assertion holds if $ca$ is replaced by $sca$.
\end{lem}

\begin{proof}
    Clearly, $\mathcal{M}\subseteq \mathcal{C}_{ca}^\circ$, and solidness, convexity, and $\sigma(ca_c,L^\infty_c)$-closedness is shown in Corollary~\ref{cor:l0:bipolar:meas}. Suppose there is another solid, convex, and $\sigma(ca_c,L^\infty_c)$-closed subset $\mathcal{D}$ of $ca_{c+}$ such that $\mathcal{M}\subseteq \mathcal{D}\subset \mathcal{C}_{ca}^\circ$. Let $\mu\in \mathcal{C}_{ca}^\circ\setminus \mathcal{D}$. Then by the Hahn-Banach separation theorem, see \cite[Corollary 5.80]{AB2006}, there is $X\in L^\infty_c$ such that
    \[\beta :=\sup_{\nu\in \mathcal{D}}\int X d\nu < \int X d\mu.\]
    Note that
    \[\beta=\sup_{\nu\in \mathcal{D}}\int X^+ d\nu\]
    where $X^+=\max\{X,0\}$. Indeed, let $A:=\{X\geq 0\}$. By solidness of $\mathcal{D}$, for all $\nu \in \mathcal{D}$ we also have $\nu_A\in \mathcal{D}$ where $\nu_A$ is given by $\nu_A[\, \cdot \,]=\nu[\, \cdot \cap A]$ ($\nu_A = 0$ in case $\nu[A] = 0$). Clearly,
    \begin{equation*}
        \int X^+d \nu = \int X d\nu_A \geq \int X d\nu.
    \end{equation*}
    Since $\int X d\mu\leq \int X^+ d\mu$, we may from now on assume that $X\in L^0_{c+}$. If $\beta=0$, then $tX\in \mathcal{C}$ for all $t>0$. However, there is $t>0$ such that $\int tXd\mu>1$, so $tX \not \in \mathcal{C}_{ca}^{\circ\circ}$. But this contradicts $\mathcal{C}=\mathcal{C}_{ca}^{\circ\circ}$ (see Corollary~\ref{cor:l0:bipolar:meas}). Similarly, if $\beta>0$, then $\tfrac{X}{\beta}\in \mathcal{C}$, but $\tfrac{X}{\beta}\not \in \mathcal{C}_{ca}^{\circ\circ}$ which again contradicts $\mathcal{C}=\mathcal{C}_{ca}^{\circ\circ}$. Hence, $\mu$ cannot exist.
\end{proof}

\section{Applications} \label{sec:appl}

In Sections~\ref{sec:appli:1}--\ref{sec:appli:3} we show how the bipolar theorems of \cite{GM2020}, \cite{LMS2022}, and \cite{BK2019} are special cases of our results in Section~\ref{sec:bip:thm}.

\subsection{A Robust Bipolar Theorem given in \cite{GM2020}}\label{sec:appli:1}

Our results imply the following bipolar theorem given in \cite[Theorem 14]{GM2020}:

\begin{cor}
    Assume that $ca_{c}^{*} = L^{\infty}_{c}$, i.e., the norm dual space of $ca_{c}$ can be identified with $L^\infty_c$. Let $\mathcal{C} \subseteq L^{0}_{c+}$ be non-empty, convex, order closed, and solid in $L^{0}_{c+}$. Set
    \begin{equation*}
        ca_{c}^{\infty} := \mathrm{span}\{\mu_{P, Z} \colon P \in \mathcal{P}, \ Z \in L^{\infty}_{c}\},
    \end{equation*}
    the linear space spanned by signed measures of type $\mu_{P, Z}[A] := E_{P}[Z \mathbf{1}_{A}]$, $A \in \mathcal{F}$. Then we have
    \begin{equation*}
        \mathcal{C} = \mathcal{C}^{**} := \bigg\{X \in L^{0}_{c+} \colon \forall \mu \in \mathcal{C}^{*} \ \int Xd\mu \leq 1\bigg\},
    \end{equation*}
    where
    \begin{equation*}
        \mathcal{C}^{*} := \bigg\{\mu \in ca_{c+}^{\infty} \colon \forall X \in \mathcal{C} \ \int Xd\mu \leq 1\bigg\}.
    \end{equation*}
\end{cor}
\begin{proof}
    The condition $ca_{c}^{*} = L^{\infty}_{c}$ implies that $\mathcal P$ is of class (S)  (see \cite[Lemma 5.15]{LMS2022}) and that $sca_c=ca_c$ (see \cite[Theorem 4.60]{AB2006}). Therefore, in particular, $ca_{c}^{\infty}\subseteq sca_c$. As $ca_{c}^{\infty}$ is separating the points of $L^\infty_c$, Corollary~\ref{cor:oCLsd:sensi} implies that $\mathcal C$ is $\mathcal P$-sensitive with reduction set $ \mathcal{Q}:=ca_{c+}^{\infty}\cap \mathfrak{P}_c(\Omega)$.
    Let $(Q,Z)$ be an element of the polar $\mathcal{C}^{\circ}_{\mathcal{Q}}$ given in Theorem~\ref{thm:l0:bipolar:ext:ks}. The condition 
    \begin{equation*}
        \forall X\in \mathcal{C} \quad E_Q[XZ] \leq 1
    \end{equation*}
    is equivalent to
    \begin{equation*}
        \forall X\in \mathcal{C} \quad \int Xd\mu \leq 1,
    \end{equation*}
    where $\mu\in ca_{c}^{\infty}$ is given by $\mu[A]:=E_Q[\mathbf{1}_AZ]$, $A\in \mathcal{F}$, and therefore $\mu\in \mathcal{C}^\ast$.   Consequently, $\mathcal{C}^{**}\subseteq \mathcal{C}^{\circ\circ}_{\mathcal{Q}}$, and as $\mathcal{C}\subseteq \mathcal{C}^{**}$ and by Theorem~\ref{thm:l0:bipolar:ext:ks},
    $$\mathcal{C}\subseteq \mathcal{C}^{**}\subseteq \mathcal{C}^{\circ\circ}_{\mathcal{Q}}=\mathcal{C}.$$
\end{proof}

\subsection{Another Robust Bipolar Theorem provided in \cite{LMS2022}}\label{sec:appli:2}

Our results also imply the following robust bipolar theorem which can be found in \cite[Theorem 4.2]{LMS2022}:

\begin{thm}
    Suppose that $\mathcal{P}$ is of class (S). Then for all convex and solid sets $\emptyset \neq \mathcal{C} \subseteq L^{0}_{c+}$, order closedness of $\mathcal{C}$ is equivalent to $\mathcal{C} = \mathcal{C}^{\circ\circ}_{sca}$ where $\mathcal{C}^{\circ\circ}_{sca}$ is given in Corollary~\ref{cor:l0:bipolar:meas}.
\end{thm}
\begin{proof} Clearly, $\mathcal{C}^{\circ\circ}_{sca}$ is order closed (and convex and solid).
    If the convex and solid set $\mathcal{C}$ is also order closed, then, according to Corollary~\ref{cor:oCLsd:sensi}, $\mathcal{C}$ is $\mathcal{P}$-sensitive with reduction set $\mathfrak{P}_{c}(\Omega) \cap sca_{c}$. Thus, we can apply Corollary~\ref{cor:l0:bipolar:meas} to obtain the result.
\end{proof}

\subsection{Yet another Robust Bipolar Theorem given in \cite{BK2019}}\label{sec:appli:3}

Consider the case $\mathcal{P} = (\delta_{\omega})_{\omega \in \Omega}$, so that $\preceq$ coincides with the pointwise order and $L^{0}_{c}=\mathcal{L}^{0}$ and $ca_{c} = ca$. 
In \cite[Theorem 1]{BK2019} the following pointwise bipolar theorem is proved: 
\begin{thm} \label{thm:kupper:bartl}
Let $\mathcal C$ be a non-empty solid regular subset of $\mathcal{L}^{0}_+$. Then $\mathcal{C}=\mathcal{C}^{\circ\circ}_{ca}$ (where $\mathcal{C}^{\circ\circ}_{ca}$ is given in Corollary~\ref{cor:l0:bipolar:meas}) if and only if $\mathcal C$ is convex and closed under $\operatorname{lim \,  inf}$.
\end{thm}

In \cite{BK2019}, $\mathcal C$ is called regular if
\begin{equation} \label{eq:regular}
    \forall \mu\in ca_+ \quad \sup_{h\in \mathcal{C}\cap U_b}\int h d\mu = \sup_{h\in \mathcal{C}\cap C_b}\int h d\mu
\end{equation}
where $C_b$ and $U_b$ denote the spaces of bounded functions $f\in \mathcal{L}^{0}$ which are in addition continuous or upper semi-continuous, respectively. Involving continuity properties of course requires that $\Omega$ carries a topology, and in fact, \cite{BK2019} assume $\Omega$ to be a $\sigma$-compact metric space and $\mathcal{F}$ to be the corresponding Borel $\sigma$-algebra.

$\mathcal C$ is said to be closed under $\operatorname{lim \, inf}$ whenever $\operatorname{lim \,  inf}_{n\to \infty} h_n\in \mathcal{C}$ for any sequence $(h_n)_{n\in \N}\subseteq \mathcal{C}$. One verifies that, for solid sets, being closed under $\operatorname{lim \,  inf}$ is equivalent to sequential order closedness. In view of Theorem~\ref{thm:l0:bipolar:ext:ks}, we observe that the rather technical assumption of regularity \eqref{eq:regular} simply implies $\mathcal{P}$-sensitivity of $\mathcal{C}$. Indeed, recall that by Corollary~\ref{cor:l0:bipolar:meas}, $\mathcal{C}=\mathcal{C}^{\circ\circ}_{ca}$ if and only if $\mathcal{C}$ is $\mathcal{P}$-sensitive, convex, solid, and closed under $\operatorname{lim \,  inf}$ (sequentially order closed). Since, given a set $\mathcal{C}$ which is convex, solid, and closed under $\operatorname{lim \, inf}$, regularity of $\mathcal{C}$ implies $\mathcal{C}=\mathcal{C}^{\circ\circ}_{ca}$  according to Theorem~\ref{thm:kupper:bartl}, regularity must in this case imply  $\mathcal{P}$-sensitivity.

Hence, it follows that Theorem~\ref{thm:kupper:bartl} is a special case of Theorem~\ref{thm:l0:bipolar:ext:ks}. Moreover, in the following we illustrate that the necessary and sufficient requirements for $\mathcal{C}=\mathcal{C}^{\circ\circ}_{ca}$ we provide in Theorem~\ref{thm:l0:bipolar:ext:ks}, namely that $\mathcal{C}$ be $\mathcal{P}$-sensitive, convex, solid, and sequentially order closed, are weaker than replacing $\mathcal{P}$-sensitivity by regularity in the latter list of properties. To this end, we give an example of a set $\mathcal{C}$ which is $\mathcal{P}$-sensitive, convex, solid, and sequentially order closed, but not regular. 

\begin{exmp}
    Suppose that $\Omega = [0, 1]$. Let
    \begin{equation*}
        \mathcal{C} := \{X \in \mathcal{L}^{0}_{+} \colon X \preccurlyeq \mathbf{1}_{[\frac{1}{2}, 1]}\}.
    \end{equation*}
    Note that $\mathcal{C}$ is $\mathcal{P}$-sensitive (see Example~\ref{ex:sensitive:sets}~(5.)), convex, solid, and sequentially order closed. However, $\mathcal{C}$ is not regular, because $x\mapsto  \mathbf{1}_{[\frac{1}{2}, 1]}(x)$ is upper semi-continuous  and for $\mu = \delta_{\frac{1}{2}}$ we have
    \begin{equation*}
        \sup_{X \in \mathcal{C} \cap U_b} \int X d\mu = 1 > 0 = \sup_{X \in \mathcal{C} \cap C_b} \int X d\mu.
    \end{equation*}
\end{exmp}

\subsection{Superhedging and Martingale Measures} \label{sec:superhedging}

Recall Example~\ref{exmp:pSensi:aggre} and the set of superhedgeable claims at cost less than $1$
\begin{equation*}
    \mathcal{C} = \{X \in L^{0}_{c+} \colon \exists H \in \mathcal{H} \ X \preccurlyeq 1 + (H \cdot S)_{T}\}.
\end{equation*}
Clearly, $\mathcal{C}$ is non-empty, convex, and solid. Suppose that $\mathcal{C}$ is also $\mathcal P$-sensitive, see Example~\ref{exmp:pSensi:aggre}, and sequentially order closed. Then, according to Corollary~\ref{cor:l0:bipolar:meas}, $\mathcal{C} = \mathcal{C}^{\circ\circ}_{ca}$. We recall that under some conditions on $S$ and $\mathcal H$ the set $\mathcal{C}^{\circ}_{ca}\cap \mathfrak{P}_c(\Omega)$ is well-known to be closely related to the set of (local) martingale measures for $S$ (see e.g.\ \cite{KS1999} for the dominated case): For illustration, suppose that $S$ is a one-dimensional continuous process adapted to a filtration $(\mathcal{F}_t)_{t\geq 0}$, and let $\mathcal{P}_{b,sem}\subseteq \mathfrak{P}_c(\Omega)$ be the set of probability measures such that $S$ is a bounded semi-martingale under each $Q\in \mathcal{P}_{b,sem}$. Let $\mathcal{H}$ be (a subset of) the set of all processes such that the stochastic integral $(H \cdot S)={}^{Q}(H\cdot S)$ is defined for all $Q\in \mathcal{P}_{b,sem}$ in the usual semi-martingale sense. Further, let  $\mathcal{P}\subseteq \mathcal{P}_{b,sem}$.  The condition $X \preccurlyeq 1 + (H \cdot S)_{T}$ can then be interpreted as $P[X\leq 1 + {}^{P}(H\cdot S)_T] = 1$ for all $P\in \mathcal{P}$. One can go further and require that there is a universal process $(H \cdot S)$ which coincides $Q$-a.s.\ with the stochastic integrals ${}^{Q}(H\cdot S)$ for all $Q\in \mathcal{P}_{b,sem}$. This can be achieved for cadlag integrands $H\in \mathcal{H}$ according to \cite{K1995}. The probability model typically chosen here is the Wiener space with $S$ being the canonical process, and $(\mathcal{F}_t)_{t\geq 0}$ being the canonical filtration (or completions of that, appropriate for various purposes), see, for instance, \cite{STZ2011} or \cite{BKN2021}.

Let us verify that any $Q\in \mathcal{C}^{\circ}_{ca}\cap \mathfrak{P}_c(\Omega)$ is a martingale measure for $S$: Note that for simple processes of type $H^{A,t}_a(s,\omega):=a1_A(\omega)1_{(t,T]}(s)$, where  $A\in \mathcal{F}_t$, $a>0$, and $t\in [0,T]$, the (universal) process
\begin{equation*}
    (H^{A,t}_a\cdot S)_u= a 1_A(S_{T\wedge u}-S_{t\wedge u})
\end{equation*}
satisfies $(H^{A,t}_a\cdot S)={}^{P}(H^{A,t}_a\cdot S)$ $P$-a.s.\ for all $P\in \mathcal{P}_{b,sem}$. By boundedness of $S$, we find $a>0$ such that
\begin{equation*}
    -1\preccurlyeq (H_a^{A,t} \cdot S)_{T}= a1_A(S_T-S_t)\preccurlyeq 1.
\end{equation*}
Hence, $1+(H_a^{A,t} \cdot S)_{T}\in \mathcal{C}$ and $1-(H_a^{A,t} \cdot S)_{T}\in \mathcal{C}$. It follows that for all $Q\in \mathcal{C}^{\circ}_{ca}\cap \mathfrak{P}_c(\Omega)$ 
\begin{equation*}
    E_Q[(H_a^{A,t} \cdot S)_{T}]=0. 
\end{equation*}
That implies $E_Q[1_A(S_T-S_t)]=0$ and hence the martingale property of $S$ under $Q$.

Conversely, let us show that under mild conditions $\mathcal{C}^{\circ}_{ca}\cap \mathfrak{P}_c(\Omega)$ includes all martingale measures which are dominated by some probability measure in $\mathcal{P}$. To this end, assume that the stochastic integral $(H\cdot S)$ is universally defined in the above sense, and that $\mathcal{H}$ is further restricted to those processes such that $(H\cdot S)={}^P(H\cdot S)$ is $P$-a.s.\ bounded from below for any $P\in \mathcal{P}$, where the bound may depend on $H$ and $P$.  Consider a martingale measure $Q\in \mathfrak{P}_c(\Omega)$ for $S$ such that $Q\ll P$ for some $P\in \mathcal{P}$. The stochastic integrals $(H\cdot S)={}^{Q}(H\cdot S)$ are local martingales under $Q$ and the required $P$-a.s.\ lower bound for $(H\cdot S)$ is also a $Q$-a.s.\ lower bound for $(H\cdot S)$. Thus, the processes $(H\cdot S)$, $H\in \mathcal{H}$, are in fact $Q$-supermartingales. Hence, for all $H\in \mathcal{H}$,
\[E_{Q}[(H \cdot S)_{T}]\leq (H \cdot S)_{0}=0.\]
Consequently, for any $X\in \mathcal C$, it follows that
$E_Q[X]\leq  1$,
so $Q\in \mathcal{C}^{\circ}_{ca}\cap \mathfrak{P}_c(\Omega)$.  

A unifying study of robust fundamental theorems of asset pricing such as discussed in \cite{ABP2013}, \cite{BC2020}, \cite{BN2015}, \cite{BFH2019}, \cite{BM2020}, \cite{BRS2021}, \cite{C2024}, \cite{OW2021}, and \cite{R2015} as well as  a dual theory for robust utility maximization, adapting the considerations made in \cite{KS1999}, is work in progress.

\subsection{$\mathcal{P}$-sensitivity and  Acceptability Criteria for Random Costs/Losses} \label{sec:appli:acc}

In this example we show that $\mathcal{P}$-sensitivity is a natural property of acceptance sets in risk assessment. To this end, identify $L^0_{c+}$ with random costs/losses. Consider a non-empty set  $\mathcal A\subseteq L^0_{c+}$ of acceptable random costs. Assuming that $\mathcal A$ is solid means that if some costs are acceptable, then less costs are too. Convexity means that cost diversification is not penalized, and sequential order closedness implies that for an order convergent increasing sequence of acceptable losses, the limit remains acceptable. Finally, consider $\mathcal P$-sensitivity: Suppose that $\mathcal{A}$ is also $\mathcal P$-sensitive with reduction set $\mathcal{Q}$.
Then Theorem~\ref{thm:l0:bipolar:ext:ks} provides a dual characterisation of acceptability
\begin{equation*}
    X \in \mathcal A \quad \Longleftrightarrow \quad X \in L^0_{c+} \wedge \sup_{(Q,Z) \in \mathcal{A}_{\mathcal Q}^\circ} E_Q[Z X] \leq 1.
\end{equation*}
The interpretation of $\mathcal{A}_{\mathcal Q}^\circ$ is clear: Whether a loss $X$ is acceptable depends on a number of probability models $\mathcal{Q}$ under which $X$ is (stress) tested. Under each $Q\in \mathcal{Q}$, $X$ has to meet the requirement that $E_Q[ZX]\leq 1$ for the model specific test functions $T_Q := \{Z \colon (Z, Q) \in \mathcal{A}_{\mathcal Q}^\circ\}$. Defining the local acceptance sets $\mathcal{A}_Q := \{X \in L^0_{Q+} \colon \forall Z \in T_Q \ E_Q[j_Q(Z) X] \leq 1\}$, $Q \in \mathcal{Q}$, we have that
\begin{equation} \label{eq:acc}
    \mathcal{A} = L^0_{c+}\cap\bigcap_{Q\in \mathcal{Q}} j_Q^{-1}(\mathcal{A}_Q).
\end{equation} 

Conversely, a natural approach to robust risk assessment is to fix a set of probability measures $\mathcal{Q}\subseteq \mathfrak{P}_c(\Omega)$ and to evaluate the risk of a loss $X$ under each model $Q\in \mathcal{Q}$, for instance, by verifying whether $j_Q(X)\in \mathcal{A}_Q$ for a set of $Q$-acceptable losses $\mathcal{A}_Q\subseteq L^0_{Q}$. In that case the acceptable losses, these are those $X\in L^0_{c+}$ which are acceptable under each $Q\in \mathcal{Q}$, satisfy $$X\in \mathcal{A}=L^0_{c+} \cap \bigcap_{Q\in \mathcal{Q}} j_Q^{-1}(\mathcal{A}_Q).$$ Hence, the acceptance set $\mathcal{A}$ is by construction of type \eqref{eq:acc}, and therefore $\mathcal{P}$-sensitive according to Lemmas~\ref{lem:Psensi:acc} and \ref{lem:intersect:Psensi}. 

This illustrates that $\mathcal{P}$-sensitivity is indeed a quite natural requirement for robust risk assessment, since it corresponds to acceptability criteria of type \eqref{eq:acc}, which are based on evaluating the risk under different possible probability models and then taking a worst-case approach. 

\subsection{A Mass Transport Type Duality}\label{sec:mass:transport}

This application is inspired by \cite{BK2019} and a straightforward generalization of \cite[Section 4]{BK2019}. Consider two measurable spaces $(\Omega_{1},\mathcal{F}_1)$ and $(\Omega_{2},\mathcal{F}_2)$. Let $\Omega:=\Omega_1\times \Omega_2$ and $\mathcal{F}:=\mathcal{F}_1\otimes \mathcal{F}_2$ denote the product space. Consider probability measures $P_1$ on $(\Omega_{1},\mathcal{F}_1)$  and  $P_2$ on $(\Omega_{2},\mathcal{F}_2)$ and the set of probability measures $\mathcal{P}$ on $(\Omega,\mathcal{F})$ consisting of all $P\in \mathfrak{P}(\Omega)$ with marginals $P[\, \cdot \times \Omega_2] = P_1$ and  $P[\Omega_1 \times \cdot \,] = P_2$. Any $f \in \mathcal{L}^0_+(\Omega)$, which serves as a goal function, gives rise to the optimal mass transport (or Monge-Kantorovich) problem
\begin{equation*}
    \int f d P  \to \max \quad \text{subject to}\quad P\in \mathcal{P}.
\end{equation*} 
In fact, as we have been practising so far, we may identify $f$ with the equivalence class $X=[f]_c$ generated by $f$ in $L^0_c(\Omega)$ and write
\begin{equation*}
    \int X d P \to \max \quad \text{subject to}\quad P\in \mathcal{P}
\end{equation*}
where $c[A]=\sup_{P \in \mathcal P} P[A]$, $A \in \mathcal{F}$, is the upper probability corresponding to $\mathcal P$ on the product space $(\Omega, \mathcal{F})$, and  $L^0_c(\Omega)$ is the space of equivalence classes of $\mathcal P$-q.s.\ equal random variables on $(\Omega,\mathcal{F})$. \\
A robustification  of this problem is obtained by replacing the marginals $P_1$ and $P_2$ with  sets of marginals $\mathcal{P}_1\subseteq \mathfrak{P}(\Omega_1)$ and $\mathcal{P}_2\subseteq \mathfrak{P}(\Omega_2)$. Now $\mathcal{P}$ is the set of all probability measures $\mathcal{P}$ on $(\Omega,\mathcal{F})$ such that $P[ \, \cdot \times \Omega_2] \in \mathcal{P}_1$ and  $P[\Omega_1 \times \cdot \, ] \in \mathcal{P}_2$. We thus obtain the upper probabilities 
\begin{equation*}
    c_1[A] = \sup_{P \in \mathcal{P}_1} P[A], \, A \in \mathcal{F}_1, \quad  c_2[A] = \sup_{P \in \mathcal{P}_2} P[A], \, A \in \mathcal{F}_2, \quad \text{and} \quad c[A] = \sup_{P \in \mathcal{P}} P[A], \, A \in \mathcal{F},
\end{equation*}
and the corresponding spaces $L^0_{c_1}(\Omega_1)$, $L^0_{c_2}(\Omega_2)$, and $L^0_{c}(\Omega)$.
For $X_1\in L^0_{c_1}(\Omega_{1})$ and  $X_2\in L^0_{c_2}(\Omega_{2})$ we write $X_1\oplus X_2$ for the $\mathcal P$-q.s.\ equivalence class in $L^0_c(\Omega)$ given by $f_{1} \oplus f_{2}(\omega) := f_{1}(\omega_{1}) + f_{2}(\omega_{2})$, $\omega=(\omega_1,\omega_2)\in \Omega_1\times \Omega_2$, where $f_{1} \in X_{1}$ and $f_{2} \in X_{2}$.

Unfortunately, before we can state our duality result, we have to relax the mass transport problem as follows: Let $\mathcal{M}_i\subseteq ca_{c_i+}(\Omega_i)$ be a set such that  $\mathcal{M}_i=\mathcal{C}_{i,ca}^\circ$ for some non-empty, convex, solid, $\mathcal{P}_i$-sensitive, and sequentially order closed sets $\mathcal{C}_i\subseteq L^0_{c_i+}(\Omega_i)$, i=1,2. This assumption is needed to apply Theorem~\ref{thm:l0:bipolar:ext:ks} in the version of Corollary~\ref{cor:l0:bipolar:meas} in the proof of Theorem~\ref{thm:optimal:transport} below. Then, for $X\in L^0_{c+}$, we consider the problem
\begin{equation} \label{eq:optimal:whatever}
    \int X d \mu  \to \max \quad \text{subject to}\quad \mu\in \mathcal{M}
\end{equation}
where $\mathcal{M}\subseteq ca_{c+}(\Omega) $ is the set of finite measures $\mu$ on $(\Omega, \mathcal F)$ such that  the marginals satisfy $\mu[ \, \cdot \times \Omega_2] \in \mathcal{M}_1$ and  $\mu[\Omega_1 \times \cdot \, ] \in \mathcal{M}_2$. The dual problem to \eqref{eq:optimal:whatever} is given by 
\begin{equation} \label{eq:optimal:transport:dual}
    \sup_{\mu_1 \in \mathcal{M}_1} \int X_1 d\mu_1 + \sup_{\mu_2 \in \mathcal{M}_2} \int X_2 d\mu_2 \to \min \quad \text{subject to} \quad (X_1, X_2) \in \Psi_X
\end{equation}
where
\begin{equation*}
    \Psi_X := \{(X_1, X_2) \in L^0_{c_1+}(\Omega_1) \times L^0_{c_2+}(\Omega_2) \colon X \preccurlyeq_{\mathcal{P}} X_1 \oplus X_2\}.
\end{equation*}
Suppose that the problem \eqref{eq:optimal:whatever} is non-trivial in the sense that $\sup_{\mu\in \mathcal{M}}\int X d\mu>0$. Further, suppose that \eqref{eq:optimal:whatever} is well-posed in the sense that $\sup_{\mu\in \mathcal{M}}\int X d\mu<\infty$. Then, after a suitable normalization, we may assume that $\sup_{\mu\in \mathcal{M}}\int X d\mu=1$. Hence, $X$ is an element of the following set
\begin{equation*}
    \mathcal{D} := \bigg\{Y\in L^0_{c+}(\Omega) \colon \sup_{\mu\in \mathcal{M}} \int Y d\mu\leq 1\bigg\}.
\end{equation*}
Consider the set
\begin{equation} \label{eq:C}
    \mathcal{C} := \bigg\{Y \in L^0_{c+}(\Omega) \colon \exists (Y_1, Y_2) \in \Psi_Y \ \sup_{\mu_1 \in \mathcal{M}_1} \int Y_1 d\mu_1 + \sup_{\mu_2 \in \mathcal{M}_2} \int Y_2 d\mu_2 \leq 1\bigg\}.
\end{equation}
By monotonicity of the integral, we have that $\mathcal{C} \subseteq \mathcal{D}$. If we are able to show that $\mathcal{C}=\mathcal{D}$, then there is $(X_1,X_2)\in \Psi_X$ such that
\begin{equation*}
    1 \geq \sup_{\mu_1 \in \mathcal{M}_1} \int X_1 d\mu_1 + \sup_{\mu_2 \in \mathcal{M}_2} \int X_2 d\mu_2 \geq \sup_{\mu \in \mathcal{M}} \int X d\mu = 1.
\end{equation*}
In other words, the dual problem \eqref{eq:optimal:transport:dual} admits a solution $(X_1,X_2)$ and there is no duality gap, i.e., $$\min_{(X_1,X_2)\in \Psi_X}\sup_{\mu_1\in \mathcal{M}_1} \int X_1 d\mu_1 + \sup_{\mu_2\in \mathcal{M}_2}\int X_2 d\mu_2  = \sup_{\mu\in \mathcal{M}}\int X d\mu. $$

\begin{thm} \label{thm:optimal:transport}
    $\mathcal{C}^{\circ \circ}_{ca} = \mathcal{D}^{\circ \circ}_{ca} = \mathcal{D}$. $\mathcal{C}=\mathcal{D}$ if and only if $\mathcal{C}$ is $\mathcal{P}$-sensitive and sequentially order closed.    
\end{thm}

Before we prove Theorem~\ref{thm:optimal:transport}, consider the following auxiliary lemma.

\begin{lem} \label{lem:aux:mass}
    For $\mu\in ca_{c+}(\Omega)$ we denote by $\mu_1[\, \cdot \,] := \mu[\, \cdot \times \Omega_2] \in ca_{c_1+}(\Omega_1)$ and $\mu_2[\, \cdot \,] =\mu[\Omega_1 \times \cdot \,] \in ca_{c_2+}(\Omega_2)$ the corresponding marginal distributions. Then, 
    \begin{equation} \label{eq:mass:1}
        \sup_{X \in \mathcal{C}} \int X d\mu = \max_{i \in \{1, 2\}} \sup_{X_{i} \in \mathcal{C}_{i}} \int X_{i} d\mu_{i}.
    \end{equation}
    Consequently,
    \begin{equation*}
        \mathcal{C}_{ca}^{\circ} = \big\{\mu \in ca_{c+}(\Omega) \colon \mu_{i} \in \mathcal{M}_{i}, \, i \in \{1, 2\}\big\}=\mathcal{M}.
    \end{equation*}
\end{lem}
\begin{proof}
    Consider $X \in \mathcal{C}$ and let $(X_{1}, X_{2}) \in \Psi_X$ such that
    \begin{equation*}
        \sup_{\nu_1 \in \mathcal{M}_1} \int X_1 d\nu_1 + \sup_{\nu_2 \in \mathcal{M}_2} \int X_2 d\nu_2 \leq 1.
    \end{equation*}
    Suppose that $\sup_{\nu_i\in \mathcal{M}_i}\int X_id\nu_i>0$, $i=1,2$, then 
    \begin{align*}
        \int X d\mu &\leq \int X_{1} \oplus X_{2} d\mu \\ &= \int X_{1} d\mu_{1} + \int X_{2} d\mu_{2} \\
        &= \sup_{\nu_1\in \mathcal{M}_1}\int X_1d\nu_1 \int \frac{X_{1}}{\sup\limits_{\nu_1\in \mathcal{M}_1}\int X_1d\nu_1} d\mu_{1} + \sup_{\nu_2\in \mathcal{M}_2}\int X_2d\nu_2\int \frac{X_{2}}{\sup\limits_{\nu_2\in \mathcal{M}_2}\int X_2d\nu_2} d\mu_{2} \\
        &\leq  \sup_{\nu_1\in \mathcal{M}_1}\int X_1d\nu_1  \sup_{Y_{1} \in \mathcal{C}_{1}} \int Y_{1} d\mu_{1} + \sup_{\nu_2\in \mathcal{M}_2}\int X_2d\nu_2  \sup_{Y_{2} \in \mathcal{C}_{2}} \int Y_{2} d\mu_{2} \\
        &\leq \max_{i \in \{1, 2\}} \sup_{Y_{i} \in \mathcal{C}_{i}} \int Y_{i} d\mu_{i}
    \end{align*}
    where, for the second inequality, we used Corollary~\ref{cor:l0:bipolar:meas} to infer that
    \begin{equation*}
        \frac{X_{i}}{\sup\limits_{\nu_i \in \mathcal{M}_i} \int X_{i} d\nu_i} \in \mathcal{C}_{i, ca}^{\circ\circ} = \mathcal{C}_i, \quad i=1, 2.
    \end{equation*}
    If $\sup_{\nu_i\in \mathcal{M}_i}\int X_id\nu_i=0$, then $X_{i}/t \in \mathcal{C}_{i}$ for all $t > 0$ by Corollary~\ref{cor:l0:bipolar:meas}. Without loss of generality assume now that $\sup_{\nu_1\in \mathcal{M}_1}\int X_1 d\nu_1 = 0$. We then have that $\sup_{\nu_2\in \mathcal{M}_2}\int X_2 d\nu_2 \leq 1$ and therefore $X_{2} \in \mathcal{C}^{\circ\circ}_{2, ca}=\mathcal{C}_{2}$  (see Corollary~\ref{cor:l0:bipolar:meas}). Thus, for all $t > 0$,
    \begin{align*}
        \int X d\mu &\leq  \int X_{1} d\mu_{1} + \int X_{2} d\mu_{2}\\ & = t \int \frac{1}{t} X_{1} d\mu_{1} + \int X_{2} d\mu_{2} \\
        &\leq t \sup_{Y_{1} \in \mathcal{C}_{1}} \int Y_{1} d\mu_{1} + \sup_{Y_{2} \in \mathcal{C}_{2}} \int Y_{2} d\mu_{2} \\
        &\leq (1 + t) \max_{i \in \{1, 2\}} \sup_{Y_{i} \in \mathcal{C}_{i}} \int Y_{i} d\mu_{i}.
    \end{align*}
    Letting $t\to 0$ shows that indeed $\int X d\mu\leq \max_{i \in \{1, 2\}} \sup_{Y_{i} \in \mathcal{C}_{i}} \int Y_{i} d\mu_{i}$. Hence, we obtain
    \begin{equation*}
        \sup_{X \in \mathcal{C}} \int X d\mu \leq  \max_{i \in \{1, 2\}} \sup_{X_{i} \in \mathcal{C}_{i}} \int X_{i} d\mu_{i}.
    \end{equation*} 
    In order to show the reverse inequality, for $X_1\in \mathcal{C}_1$ let $X:= X_{1} \oplus 0$ and for $X_2\in \mathcal{C}_2$ let $\tilde{X}=0\oplus X_2$. Then, using the fact that $C_i=C_{i,ca}^{\circ\circ}$ by Corollaries~\ref{cor:them:KS:ext} and \ref{cor:l0:bipolar:meas} another time, we infer that $X,\tilde X \in \mathcal{C}$. Moreover,
    \begin{equation*}
        \int X_{1} d\mu_{1} = \int X d\mu \leq \sup_{Y \in \mathcal{C}} \int Y d\mu \quad \text{and} \quad \int X_{2} d\mu_{2} = \int \tilde{X} d\mu \leq \sup_{Y \in \mathcal{C}} \int Y d\mu.
    \end{equation*}
    It follows that
    \begin{equation*}
        \max_{i \in \{1, 2\}}   \sup_{X_{i} \in \mathcal{C}_{i}} \int X_{i} d\mu_{i} \leq \sup_{X \in \mathcal{C}} \int X d\mu,
    \end{equation*}
    and thus \eqref{eq:mass:1} is proved. Finally,
    \begin{align*}
        \mathcal{C}_{ca}^{\circ} &= \bigg\{\mu \in ca_{c+}(\Omega) \colon \forall X \in \mathcal{C} \ \int X d\mu \leq 1\bigg\} \\
        &= \bigg\{\mu \in ca_{c+}(\Omega) \colon \sup_{X \in \mathcal{C}} \int X d\mu \leq 1\bigg\} \\
        &= \bigg\{\mu \in ca_{c+}(\Omega) \colon \max_{i \in \{1, 2\}} \sup_{X_{i} \in \mathcal{C}_{i}} \int X_{i} d\mu_{i} \leq 1\bigg\} \\
        &= \big\{\mu \in ca_{c+}(\Omega) \colon \mu_{i} \in \mathcal{C}_{i, ca}^{\circ}, \, i \in \{1, 2\}\big\}=\mathcal{M}.
    \end{align*}    
\end{proof}  

\begin{cor} \label{cor:D:C:polar}
    $\mathcal{C}_{ca}^{\circ} =\mathcal{M}=\mathcal{D}^\circ_{ca}$. 
\end{cor}
\begin{proof}
    $\mathcal{C} \subseteq \mathcal{D}$, Lemma~\ref{lem:polar:aux}, and the definition of $\mathcal D$ imply that $\mathcal{M}\subseteq \mathcal{D}_{ca}^{\circ}\subseteq \mathcal{C}_{ca}^{\circ}$. According to Lemma~\ref{lem:aux:mass}, $\mathcal{M}=\mathcal{C}^\circ_{ca}$. 
\end{proof}
    
\begin{proof}[Proof of Theorem~\ref{thm:optimal:transport}]
    $\mathcal D$ is non-empty, convex, solid, $\mathcal P$-sensitive, and sequentially order closed by definition (see also Proposition~\ref{prop:pSensi}). By Corollary~\ref{cor:l0:bipolar:meas}, it thus holds that $\mathcal{D} = \mathcal{D}^{\circ \circ}_{ca}$. $\mathcal{C}^{\circ \circ}_{ca}=\mathcal{D}^{\circ \circ}_{ca}=\mathcal{D}$ then follows from Corollary~\ref{cor:D:C:polar}. \\    
    Clearly, if $\mathcal{C}=\mathcal{D}$, then $\mathcal{C}$ inherits the properties $\mathcal{P}$-sensitivity and sequential order closedness from $\mathcal{D}$. Now suppose that $\mathcal{C}$ is $\mathcal P$-sensitive and sequentially order closed. It is clear that $\mathcal{C}$ is also non-empty, convex, and solid.  Hence, by Corollary~\ref{cor:l0:bipolar:meas}, we have $\mathcal{C} = \mathcal{C}_{ca}^{\circ\circ}$.
\end{proof}

The following examples illustrate Theorem~\ref{thm:optimal:transport}. By definition of the set $\mathcal{C}$ it is clear that, in general, $\mathcal{P}$-sensitivity and sequential order closedness may be challenging to verify. In Example~\ref{ex:mass:1} finiteness of one of the spaces $\Omega_i$ will make this possible. In Example~\ref{ex:mass:2}, we give an example where $\mathcal{C}=\mathcal{D}$ is easily directly verified, while showing that $\mathcal{C}$ in its representation \eqref{eq:C} is sequentially order closed and $\mathcal{P}$-sensitive seems more challenging. 

However, in any case  Theorem~\ref{thm:optimal:transport} is interesting from a structural point of view. It shows that if things behave nicely---that is, the dual problem admits solutions for all $X\in L^0_{c+}$ such that $\sup_{\mu \in \mathcal{M}}\int X d\mu\leq 1$ and there is no duality gap---then this requires $\mathcal{P}$-sensitivity of $\mathcal{C}$ (and sequential order closedness), again highlighting $\mathcal{P}$-sensitivity as a structural condition which makes robust models manageable. 

\begin{exmp} \label{ex:mass:1}
    Let
    \begin{equation*}
        \mathcal{C}_1 := \{X \in L^0_{c_1+}(\Omega_1) \colon X \mathbf{1}_{A_1} = 0\},
    \end{equation*}
    where $A_1 \in \mathcal{F}_1$ satisfies $c_1[A_1] > 0$ and $c_1[A^c_1] > 0$, and let 
    \begin{equation*}
        \mathcal{C}_2 := \{X \in L^0_{c_2+}(\Omega_2) \colon X \preccurlyeq_{\mathcal{P}_2} 1\}.
    \end{equation*}
    Clearly, another representation of $\mathcal{C}_1$ is $\mathcal{C}_1 = \{X \mathbf{1}_{A_1^c} \colon X \in L^0_{c_1+}(\Omega_1)\}$. 
    Since $\mathcal{C}_1$ is a cone, one verifies that
    \begin{equation*}
        \mathcal{M}_1 = \mathcal{C}_{1, ca}^\circ = \{\mu_1 \in ca_{c_1+}(\Omega_1) \colon \mu_1[A^c_1] = 0\}.
    \end{equation*}
    Regarding $\mathcal{M}_2$, we have
    \begin{equation*}
        \mathcal{M}_2 = \mathcal{C}_{2,ca}^\circ = \{\mu_2 \in ca_{c_2+}(\Omega_2) \colon \mu_2[\Omega_2] \leq 1\}.
    \end{equation*}
    Let us, from now on, assume that $\Omega_2$ is finite and that $\mathcal{F}_2$ is the power set of $\Omega_2$. Moreover, we assume that $\mathcal{P}$ is of class (S). We will show that $\mathcal{C}$ is order closed, and thus deduce $\mathcal{P}$-sensitivity by Corollary~\ref{cor:oCLsd:sensi}. To this end, let the net $(Y_\alpha)_{\alpha\in I}\subseteq\mathcal{C}$ satisfy $Y_\alpha \order{c} Y$ where $Y\in L^0_{c}$. 
    Obviously, $0\preccurlyeq_{\mathcal{P}} Y$. Let $(Y_{1,\alpha},Y_{2,\alpha})\in \Psi_{Y_\alpha}$ such that
    \begin{equation} \label{eq:mass:ex2}
        \sup_{\mu_1\in \mathcal{M}_1} \int Y_{1,\alpha} d\mu_1 + \sup_{\mu_2\in \mathcal{M}_2}\int Y_{2,\alpha} d\mu_2 \leq 1.
    \end{equation}
    Since $\mathcal{M}_1$ is a cone, \eqref{eq:mass:ex2} is satisfied only if $Y_{1,\alpha}\in \mathcal{C}_1$, and thus, $\sup_{\mu_1 \in \mathcal{M}_1} \int Y_{1, \alpha} d\mu_1 = 0$. 
    This implies that $Y_{2, \alpha}\in \mathcal{C}_2$. Conversely, any pair $(Y_1,Y_2)\in \mathcal{C}_1\times \mathcal{C}_2$ satisfies \eqref{eq:mass:ex2}. 
    Define (identifying equivalence classes and their representatives)
    \begin{equation*}
        Y_1(\cdot) := \max_{\omega_2 \in \Omega_2} Y(\cdot, \omega_2) \mathbf{1}_{A^c_1} \quad \text{and} \quad Y_2 := \sup_{\alpha \in I} Y_{\alpha, 2}
    \end{equation*}
    in $L^0_{c_1}(\Omega_1)$ and  $L^0_{c_2}(\Omega_2)$, respectively. Then $(Y_1,Y_2)\in \mathcal{C}_1\times \mathcal{C}_2$, and thus \eqref{eq:mass:ex2} is satisfied. Moreover,
    \begin{equation*}
        Y\mathbf{1}_{A_1^c\times \Omega_2} \preccurlyeq_{\mathcal{P}} Y_1\oplus 0 \preccurlyeq_{\mathcal{P}} (Y_1\oplus Y_2)\mathbf{1}_{A_1^c\times \Omega_2},
    \end{equation*}
    and, for $\alpha\in I$,
    \begin{equation*}
        Y_\alpha \mathbf{1}_{A_1\times \Omega_2}\preccurlyeq_{\mathcal{P}} (Y_{1,\alpha} \oplus Y_{2,\alpha})\mathbf{1}_{A_1\times \Omega_2} = (0\oplus Y_{2,\alpha}) \mathbf{1}_{A_1\times \Omega_2} \preccurlyeq_{\mathcal{P}} (0\oplus Y_2) \mathbf{1}_{A_1\times \Omega_2}=(Y_1\oplus Y_2) \mathbf{1}_{A_1\times \Omega_2},
    \end{equation*}
    where we used that $Y_1\mathbf{1}_{A_1}=0$ and $Y_{1,\alpha}\mathbf{1}_{A_1}=0$ for all $\alpha\in I$. We conclude that $Y\preccurlyeq_{\mathcal{P}} Y_1\oplus Y_2$, and therefore $Y\in \mathcal{C}$. Hence, $\mathcal{C}$ is order closed and Theorem~\ref{thm:optimal:transport} yields $\mathcal{D}=\mathcal{C}$. 
\end{exmp}

\begin{exmp} \label{ex:mass:2} 
    Let
    \begin{equation*}
        \mathcal{C}_i := \{X \in L^0_{c_i+}(\Omega_i) \colon X \preccurlyeq_{\mathcal{P}_i} 1\}, \quad i=1, 2.
    \end{equation*}
    One verifies that
    \begin{equation*}
        \mathcal{M}_i = \mathcal{C}_{i,ca}^\circ = \{\mu_i \in ca_{c_i+}(\Omega_i) \colon \mu_i[\Omega] \leq 1\}, \quad i=1, 2.
    \end{equation*}
    Hence,
    \begin{equation*}
        \mathcal{M} = \{\mu\in ca_{c+}(\Omega) \colon \mu[\Omega] \leq 1\},
    \end{equation*}
    and   
    \begin{equation*}
        \mathcal{D} = \{X \in L^0_{c+}(\Omega) \colon X \preccurlyeq_{\mathcal{P}} 1\}.
    \end{equation*}
    Now consider $\mathcal{C}$. Let $Y\in L^0_{c+}$ such that $Y\preccurlyeq_{\mathcal{P}} 1$. Then $Y_i\in L^0_{c_i+}$ given by $Y_i\equiv 1/2$ trivially satisfy $Y\preccurlyeq_{\mathcal{P}} Y_1\otimes Y_2$, that is $(Y_1,Y_2)\in \Psi_Y$, and 
    \begin{equation*}
        \sup_{\mu_1\in \mathcal{M}_1} \int Y_1 d\mu_1 + \sup_{\mu_2\in \mathcal{M}_2}\int Y_2 d\mu_2 = 1.
    \end{equation*}
    Hence, $\mathcal{D}\subseteq \mathcal{C}$, and recalling that $\mathcal{C}\subseteq \mathcal{D}$, we have $\mathcal{D}= \mathcal{C}$. However, concluding this from Theorem~\ref{thm:optimal:transport}, that is  verifying sequential order closedness and $\mathcal{P}$-sensitivity of $\mathcal{C}$ in the representation \eqref{eq:C} does not seem trivial.
\end{exmp}

% \bibliographystyle{agsm}
% \bibliography{references}
\printbibliography

\end{document}